\documentclass[11pt]{amsart}

\usepackage{latexsym, mathrsfs, color, tikz, multirow,bbm,mathtools}
\usepackage{graphics}
\usepackage[flushleft]{threeparttable}
\usepackage{subcaption}
\usepackage{xparse}
\input{xy}
\xyoption{all}
\usepackage[left=1in,top=1in,bottom=1in,right=1in]{geometry} 

\usepackage{tikz}
\usepackage{tikz-cd}

\numberwithin{equation}{section}
\theoremstyle{plain}
\newtheorem{Thm}[equation]{Theorem}
\newtheorem{Prop}[equation]{Proposition}
\newtheorem{Cor}[equation]{Corollary}
\newtheorem{Lem}[equation]{Lemma}

\newtheorem{Quest}[equation]{Question}

\theoremstyle{definition}
\newtheorem{Def}[equation]{Definition}
\newtheorem{Exa}[equation]{Example}
\newtheorem{Rmk}[equation]{Remark}

\newenvironment{red}{\relax\color{red}}{\relax}
\newenvironment{blue}{\relax\color{blue}}{\hspace*{.5ex}\relax}

\newcommand{\ber}{\begin{red}}
\newcommand{\er}{\end{red}}
\newcommand{\beb}{\begin{blue}}
\newcommand{\eb}{\end{blue}}
\def\KL#1{(\textcolor{purple}{KL:#1})}

\newcommand{\bw}{{\boldsymbol{w}}}

\DeclareMathOperator{\sgn}{sgn}

\makeatletter
\renewcommand*\env@matrix[1][*\c@MaxMatrixCols c]{%
  \hskip -\arraycolsep
  \let\@ifnextchar\new@ifnextchar
  \array{#1}}
\makeatother

\def\KL#1{(\textcolor{purple}{KL: #1})}

\begin{document}

\title{Geometry of $C$-vectors and $C$-Matrices for Mutation-Infinite Quivers}

\author[T. J. Ervin]{Tucker J. Ervin}
\address{Department of Mathematics, University of Alabama,
	Tuscaloosa, AL 35487, U.S.A.}
 \email{tjervin@crimson.ua.edu}
\author[B. Jackson]{Blake Jackson} 
\address{Department of Mathematics, University of Connecticut,
	Monteith Bldg, 341 Mansfield Rd, Storrs, CT 06269, U.S.A.}
 \email{blake.jackson@uconn.edu}
\author[K.Lee]{Kyungyong Lee}
\address{Department of Mathematics, University of Alabama,
	Tuscaloosa, AL 35487, U.S.A. and Korea Institute for Advanced Study, Seoul 02455, Republic of Korea}
 \email{kyungyong.lee@ua.edu; klee1@kias.re.kr}
 \author[S. D. Nguyen]{Son Dang Nguyen}
 \address{Department of Radiology, University of North Carolina  at Chapel Hill}
 \email{sonng@ad.unc.edu}
 
\thanks{The authors were supported by the University of Alabama, Korea Institute for Advanced Study, and the NSF grant DMS 2042786 and DMS 2302620.}




\begin{abstract}
The set of forks is  a class of quivers introduced by M. Warkentin, where every connected mutation-infinite quiver is mutation equivalent to infinitely many forks. Let $Q$ be a fork with $n$ vertices, and $\bw$ be a fork-preserving mutation sequence.  
We show that every $c$-vector of $Q$ obtained from $\bw$ is a solution to a quadratic equation of the form 
$$\sum_{i=1}^n x_i^2 + \sum_{1\leq i<j\leq n} \pm q_{ij} x_i x_j =1,$$
where $q_{ij}$ is the number of arrows between the vertices $i$ and $j$ in $Q$.
The same proof techniques implies that when $Q$ is a rank 3 mutation-cyclic quiver, every $c$-vector of $Q$ is a solution to a quadratic equation of the same form. 
\end{abstract}




\maketitle

\section{Introduction}

 The mutation of a quiver $Q$ was discovered by S. Fomin and A. Zelevinsky in their seminal paper \cite{fomin_cluster_2002} where they introduced cluster algebras. It also appeared in the context of Seiberg duality \cite{feng_toric_2002}. 
 The $c$-vectors (and $C$-matrices) of $Q$ were defined through mutations in further developments of the theory of cluster algebras \cite{fomin_cluster_2007}, and together with their companions, $g$-vectors (and $G$-matrices), played  fundamental roles in the study of cluster algebras (for instance, see \cite{derksen_quivers_2010, gross_canonical_2018, lin_two_2023, nakanishi_tropical_2012, plamondon_cluster_2011}). 
 When $Q$ is acyclic, positive $c$-vectors are actually real Schur roots, that is, the dimension vectors of indecomposable rigid modules over $Q$ \cite{chavez_c-vectors_2015, hubery_categorification_2016, speyer_acyclic_2013}. 
 Moreover, they appear as the denominator vectors of non-initial cluster variables of the cluster algebra associated to $Q$ \cite{caldero_triangulated_2006}. 

Due to the multifaceted appearance of $c$-vectors in important constructions, there have been various results related to the description of $c$-vectors (or real Schur roots) of an acyclic quiver (\cite{baumeister_note_2014,hubery_categorification_2016,igusa_exceptional_2010,schofield_general_1992,seven_cluster_2014,speyer_acyclic_2013}).
In \cite{lee_geometric_2023}, K.-H. Lee and K. Lee conjectured  a correspondence between real Schur roots of an acyclic quiver  and  non-self-crossing curves on a Riemann surface and proposed a new combinatorial/geometric description.
The conjecture is now proven by A. Felikson and P. Tumarkin \cite{felikson_acyclic_2018} for acyclic quivers with multiple edges between every pair of vertices. Recently, S. D. Nguyen \cite{nguyen_proof_2022} proved the conjecture for an arbitrary acyclic (valued) quiver.

For a given (not necessarily acyclic) quiver $Q$, the set of quivers that are mutation equivalent to $Q$ is called the mutation equivalence class of $Q$ and denoted by $\text{Mut}(Q)$. The quiver $Q$ is said to be \emph{mutation-infinite} if $|\text{Mut}(Q)|$ is not finite, and \emph{mutation-finite} if $|\text{Mut}(Q)|<\infty$. The mutation-finite quivers are completely classified, and relatively well studied. On the other hand,  mutation-infinite quivers still await further investigations. 

A reader-friendly version of our main theorem may be stated as follows.
\begin{Thm} \label{main-thm}
Let $n$ be any positive integer. Let $P$ be a mutation-infinite connected quiver with $n$ vertices. Then there exist an infinite number of pairs of a quiver $Q\in \text{Mut}(P)$ and $k\in\{1,...,n\}$ such that every $c$-vector of $Q$ obtained from any mutation sequence not starting with $k$  is a solution to a quadratic equation of the form 
\begin{equation}\label{quad}
\sum_{i=1}^n x_i^2 + \sum_{1\leq i<j\leq n} \pm q_{ij} x_i x_j =1,
\end{equation}
where $q_{ij}$ is the number of arrows between the vertices $i$ and $j$ in $Q$.
\footnote{There does not seem to be a simple way of determining the exact signs of the $x_i x_j$ terms.}
\end{Thm}

To state a more precise theorem, we need to recall the definition of forks.
An \emph{abundant quiver} is a quiver such that there are two or more arrows between every pair of vertices.

\begin{Def}{\cite[Definition 2.1]{warkentin_exchange_2014}} \label{def-forks1}
A \textit{fork} is an abundant quiver $F$, where $F$ is not acyclic and where there exists a vertex $r$, called the point of return, such that
\begin{itemize}
    \item For all $i \in F^{-}(r)$ and $j \in F^{+}(r)$ we have $f_{ji} > f_{ir}$ and $f_{ji} > f_{rj}$, where $F^{-}(r)$ is the set of vertices with arrows pointing towards $r$ and $F^{+}(r)$ is the set of vertices with arrows coming from $r$.
    
    \item The full subquivers induced by $F^{-}(r)$ and $F^{+}(r)$ are acyclic.
\end{itemize}
An example of a fork is given by 
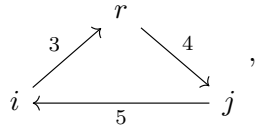
\[\begin{tikzcd}
    & r \\
    i & & j
    \arrow[from=2-1, to=1-2, "3"]
    \arrow[from=1-2, to=2-3, "4"]
    \arrow[from=2-3, to=2-1, "5"]
\end{tikzcd},\]
where $r$ is the point of return.
\end{Def}

It is known that ``most'' quivers in $\text{Mut}(Q)$ of any connected mutation-infinite quiver $Q$ are forks, as Theorem~\ref{thm-connected-fork1} and Proposition~\ref{exgraph-connected-fork} imply. 

\begin{Thm}{\textup{\cite[Theorem 3.2]{warkentin_exchange_2014}}} \label{thm-connected-fork1}
    A connected quiver is mutation-infinite if and only if it is mutation-equivalent to a fork.
\end{Thm}

\begin{Prop}{\textup{\cite[Proposition 5.2]{warkentin_exchange_2014}}} \label{exgraph-connected-fork}
Let $G$ be the exchange graph of a connected mutation-infinite quiver. A simple random walk on $G$ will almost surely leave the fork-less part and never come back.
\end{Prop}


A \emph{fork-preserving} mutation sequence is a reduced (see Definition~\ref{def_of_mut_seq})  sequence of mutations that starts with a fork and does not mutate at its point of return. 
A more precise version of our main theorem is as follows.

\begin{Thm} \label{thm-preserving-fork}
Let $Q$ be a fork, and let $\bw$ be a fork-preserving mutation sequence.  Every $c$-vector of $Q$ obtained from $\bw$  is a solution to a quadratic equation of the form \eqref{quad}.
\end{Thm}





A quiver Q is called \emph{mutation-acyclic} if it is mutation-equivalent to an acyclic quiver, else it is called \emph{mutation-cyclic}.
Ahmet Seven informed us that he had independently discovered the following theorem.

\begin{Thm} \label{rank 3}
Let $Q$ be a mutation-cyclic quiver with 3 vertices. Then every $c$-vector of $Q$  is a solution to a quadratic equation of the form 
\eqref{quad} with $n=3$.
\end{Thm}

The argument used by Ervin \cite{ervin_all_2024} motivates the following theorem. Fomin and Neville \cite[Lemma 6.14]{fomin_long_2023} also obtained a similar result.

\begin{Thm} \label{thm-base-case}
Let $\bw$ be a fork-preserving mutation sequence. 
The sign-vector (see Definition~\ref{C_sign}) of $C^\bw$ depends only on the signs of entries of initial exchange matrix $B$. 
In other words, the sign-vector is independent of the number of arrows between vertices of the initial quiver $Q$. 
\end{Thm}

\begin{Cor} \label{cor-reflection-triple}
 Let $n$ be any positive integer, and let $Q$ be a fork with $n$ vertices. 
For each fork-preserving mutation sequence $\bw$ from $Q$, the corresponding $n$-tuple of reflections $(\mathbf{r}_1^\bw, \mathbf{r}_2^\bw, \dots,  \mathbf{r}_n^\bw)$ (see Definition~\ref{def-r}) depends only on the signs of entries of the initial exchange matrix $B$.
\end{Cor}

From this, we are able to prove that the product of reflections is equal to a Coxeter element. More precisely, we have the following:

\begin{Thm} \label{thm-coxeter-element}
 Let $n$ be any positive integer, and let $Q$ be a fork with $n$ vertices. 
For each fork-preserving mutation sequence $\bw$ from $Q$, we have $$\mathbf{r}_{\lambda(1)}^\bw ... \mathbf{r}_{\lambda(n)}^\bw = \mathbf{r}_{\rho(1)}...\mathbf{r}_{\rho(n)}$$ for some permutations $\lambda, \rho \in\mathfrak{S}_n$, where $\mathfrak{S}_n$ is the symmetric group on $\{1,...,n\}$ and $\mathbf{r}_1,...,\mathbf{r}_n$ are the initial reflections.
The $\lambda$ and $\rho$ permutations are well controlled.
\end{Thm}

\begin{Cor} \label{cor-non-self-crossing}
 Let $n$ be any positive integer, and let $Q$ be a fork with $n$ vertices. 
For each fork-preserving mutation sequence $\bw$ from $Q$, there exist pairwise non-crossing and non-self-crossing admissible curves $\eta_i^\bw$ (see Definition~\ref{def-adm})  such that $\mathbf{r}_i^\bw = \nu(\eta_i^\bw)$ for every $i \in \{1,...,n\}$.
\end{Cor}

The following theorem plays a pivotal role in proving Theorem \ref{thm-l-equals-c}. 

\begin{Thm} \label{thm-linear-ordering}
 Let $n$ be any positive integer, and let $B$ be a fork with $n$ vertices. 
For each fork-preserving mutation sequence $\bw$ from $B$,  the GIM $A^\bw$ is admissible for $B^\bw$.
\end{Thm}

The following theorem implies Theorem~\ref{thm-preserving-fork}. 

\begin{Thm} \label{thm-l-equals-c}
Let $Q$ be a fork with $n$ vertices, and let $\bw$ be a fork-preserving mutation sequence. For each $i\in\{1,...,n\}$,
 there exists a diagonal matrix $D_i^\bw$ such that $(D_i^\bw)^2=1$ and $l_i^{\bw}= c_i^{\bw} D_i^\bw$ for the linear order given in Corollary \ref{cor-fork-admissible-ordering}. In other words, the entries of $l$-vectors are equal to the entries of $c$-vectors up to sign for said linear ordering. 
\end{Thm}

In the acyclic case, there is an obvious linear ordering that works for every mutation sequence, but here we only arrive at a linear ordering that depends on the mutation sequence.

\section{Preliminaries} \label{sec-preliminaries}

\subsection{$C$-matrices} Let $n$ be a positive integer. If $B = [b_{ij}]$ is an $n\times n$ skew-symmetric matrix,
then $B$ is in correspondence with a quiver $Q$ on $n$ vertices: if $b_{ij} > 0$ and $i \neq j$, then $Q$ has $b_{ij}$ arrows from vertex $i$ to vertex $j$.
The statements of some theorems have been formulated in terms of $Q$; however, we prefer to work with $B$ since the description of $c$-vectors is clearer in this setting.
Also, for a nonzero vector $c=( c_1 , \dots , c_n) \in \mathbb Z^n$, we write $c >0$ if all $c_i$ are non-negative, and  $c <0$ if all $c_i$ are non-positive. 

\begin{Def}\label{def_of_mut_seq}
Assume that $M=[m_{ij}]$ is an $n \times 2n$ matrix with integer entries. Let $\mathcal I:= \{ 1, 2, \dots, n \}$ be the set of indices. For $\bw=[i_1, i_2, \dots , i_\ell]$, $i_j \in \mathcal I$, we define the matrix $M^\bw=[m_{ij}^\bw]$ inductively: the initial matrix is $M$ for $\bw=[\,]$, and assuming we have $M^\bw$, define the matrix $M^{\bw[k]}=[m_{ij}^{\bw[k]}]$ for $k \in \mathcal I$ with $\bw[k]:=[i_1, i_2, \dots , i_\ell, k]$
by  
\begin{equation} \label{eqn-mmuu} m_{ij}^{\bw[k]} = \begin{cases} -m_{ij}^\bw & \text{if  $i=k$ or $j=k$}, \\ m_{ij}^\bw + \mathrm{sgn}(m_{ik}^\bw) \, \max(m_{ik}^\bw m_{kj}^\bw,0) & \text{otherwise}, \end{cases}
\end{equation}
where  $\mathrm{sgn}(a) \in \{1, 0, -1\}$ is the signature of $a$. The matrix $M^{\bw[k]}$ is called the {\em mutation of $M^\bw$} at index (or label) $k$, $\bw$ and $\bw[k]$ are called {\em mutation sequences}, and $n$ is the {\em rank}.
If the mutation sequence $[i_i, i_2, \cdots , i_\ell]$ satisfies $i_j\neq i_{j+1}$ for all $j\in\{1,...,\ell-1\}$, then it is said to be \emph{reduced}.
\end{Def}

Let $B$ be a $n\times n$ skew-symmetric matrix. 
Consider the $n\times 2n$ matrix $\begin{bmatrix}B&I\end{bmatrix}$ and
 a mutation sequence $\bw=[i_1, \dots ,i_\ell]$. After the mutations at the indices $i_1, \dots , i_\ell$ consecutively, we obtain $\begin{bmatrix}B^\bw&C^\bw\end{bmatrix}$. Write their entries as 
\begin{equation} \label{eqn-bcbc} B^\bw= \begin{bmatrix} b_{ij}^\bw \end{bmatrix}, \qquad
C^\bw= \begin{bmatrix} c_{ij}^\bw \end{bmatrix} = \begin{bmatrix} c_1^\bw \\ \vdots \\ c_n^\bw \end{bmatrix},\end{equation}
where $c_i^\bw$ are the row vectors.
\begin{Def}\label{C_sign}
The matrix $C^\bw$ is called a $C$-matrix of $B$ for any $w$.\footnote{This is different from the original definition by Fomin and Zelevinsky}
The row vectors $c_i^\bw$ are called {\em $c$-vectors} of $B$ for any $i$ and $\bw$. 
Each non-zero entry of $c_i^\bw$ will share the same sign \cite{derksen_quivers_2008}, allowing us to define the \emph{sign-vector} of $C^\bw$, where the $i$-th entry is 1 if $c_i^\bw \geq 0$ and $-1$ if $c_i^\bw \leq 0$.
\end{Def}

\begin{Def}
\label{def-acyclic-ordering} 
    Let $Q$ be a quiver on $n$ vertices.
    Then we say that an ordering $v_1 \prec v_2 \prec \dots \prec v_n$ on the vertices is acyclic if whenever there exists an arrow $v_i \to v_j$ then $v_i \prec v_j$. 
\end{Def}

\subsection{Reflections and $L$-matrices}

\begin{Def} \label{def-gim-and-coxeter}
    A {\em generalized intersection matrix} (GIM) is a square matrix ${A}=[a_{ij}]$ with integral entries such that
    (1) for diagonal entries, $a_{ii}=2$;
    (2) $a_{ij}>0$ if and only if $a_{ji}>0$;
    (3) $a_{ij}<0$ if and only if $a_{ji}<0$.
    
    Let $\mathcal A$ be the (unital) $\mathbb Z$-algebra generated by $s_i, e_i$, $i=1,2, \dots, n$, subject to the following relations:
    $$ s_i^2=1, \quad \sum_{i=1}^n e_i =1, \quad s_ie_i = -e_i, \quad e_is_j=  \begin{cases} s_i+e_i-1 &\text{if } i =j, \\ e_i &\text{if } i \neq j, \end{cases}  \quad e_ie_j=  \begin{cases} e_i &\text{if } i =j, \\ 0 & \text{if } i \neq j. \end{cases}$$
    Let $\mathcal W$ be the subgroup of the units of $\mathcal{A}$ generated by $s_i$, $i=1, \dots, n$. Note that $\mathcal W$ is (isomorphic to) the universal Coxeter group. 
    An element $\mathbf{r}\in \mathcal W$ is called a reflection if $\mathbf{r}^2=1$. Let $\mathfrak{R}\subset \mathcal W$ be the set of reflections.
\end{Def}

From now on, let ${A}=[a_{ij}]$ be an $n\times n$ symmetric GIM.
Let $\Gamma = \sum_{i=1}^n \mathbb{Z}\alpha_i$ be the lattice generated by the formal symbols $\alpha_1,...,\alpha_n$. Define a representation $\pi:\mathcal A \rightarrow \mathrm{End}(\Gamma)$ by
\[ \pi(s_i)(\alpha_j) = \alpha_j - a_{ji} \alpha_i \quad \text{ and } \quad \pi(e_i)(\alpha_j) =\delta_{ij} \alpha_i , \quad\text{ for }i,j\in\{1,...,n\}. \] We suppress $\pi$ when we write the action of an element of $\mathcal A$ on $\Gamma$.

Given a skew-symmetric matrix $B$, for  each  linear ordering $\prec$ on $\{1,...,n\}$, we  define the associated GIM ${A}=[a_{ij}]$ by
\begin{equation} \label{eqn-gim1}
a_{ij}= \begin{cases}  b_{ij}  & \text{ if } i \prec j , \\
2 & \text{ if } i =j, \\ -b_{ij} & \text{ if } i \succ j . \end{cases}\ 
\end{equation}
An ordering $\prec$ provides a certain way for us to regard the skew-symmetric matrix $B$ as acyclic even when it is not. 

\begin{Def} \label{def-r}
When $\bw=[\,]$, we let $\mathbf{r}_i = s_i\in \mathfrak{R}$ for each  $i\in\{1,...,n\}$. For each  mutation sequence $\bw$ and each $i\in\{1,...,n\}$, define $\mathbf{r}_i^\bw \in \mathfrak{R}$ inductively as follows: 
\begin{equation} \label{def-sx_i-1} \mathbf{r}_i^{\bw [k]}= \begin{cases} \mathbf{r}_k^\bw \mathbf{r}_i^\bw  \mathbf{r}_k^\bw & \text{ if } \ b_{ik}^\bw c_{k}^\bw>0, \\ \mathbf{r}_i^\bw & \text{ otherwise.} \end{cases} \end{equation}
Clearly, each $\mathbf{r}_i^\bw$ is written in the form
\[ \mathbf{r}_i^\bw = g_i^\bw s_i {(g_i^\bw)}^{-1}, \quad g_i^\bw \in \mathcal W, \quad i\in\{1,...,n\}.\]
\end{Def}

\begin{Def}\label{def-ell}
Let $\sgn=\{1,-1\}$ be the group of order 2, and
consider the natural group action $\sgn\times \mathbb{Z}^n\longrightarrow \mathbb{Z}^n$, where we identify $\Gamma$ with $\mathbb Z^n$. 
Choose an ordering $\prec$ on $\{1,...,n\}$ to fix a GIM ${A}$, and define
\[   l_i^\bw = g_i^\bw (\alpha_i) \in \mathbb Z^n/\sgn, \qquad i\in\{1,...,n\}, \]
where we set $\alpha_1=(1,0,...,0), ..., \ \alpha_n=(0,..., 0, 1)$. Then the {\em $L$-matrix} $L^{\bw}$ associated to ${A}$ is defined to be the $n \times n$ matrix whose $i^\text{th}$ row is $l_i^\bw$ for $i\in\{1,...,n\}$, i.e., $L^{\bw} =  \begin{bmatrix} l_1^{\bw}  \\  \vdots \\ l_n^{\bw} \end{bmatrix}$, and the vectors $l_i^{\bw}$ are called the {\em $l$-vectors of ${A}$}. Note that the $L$-matrix and $l$-vectors associated to a GIM ${A}$ implicitly depend on the representation $\pi$ which is suppressed from the notation. 
\end{Def}

\subsection{Geometry of reflections}
Here we review the definition of admissible curves \cite{lee_correspondence_2021,lee_geometric_2023}. 

Let $Q$ be a fork with $n$ vertices labeled by $I:=\{1,...,n\}$ and point of return $r$. 
Let $\sigma$ be the linear ordering given by $r \prec a_{n-1} \prec a_{n-2} \prec \dots \prec a_1$, where $a_1, a_2, \dots, a_{n-1}$ are the vertices of $Q \setminus \{r\}$ and $a_i \prec a_{i-1}$ if and only if there is an arrow from $a_{i-1}$ to $a_i$. Then $\sigma(1)=r$ and $\sigma(i)=a_{n-i+1}$ 
for $i\in\{2,...,n\}$.

We define a labeled Riemann surface  $\Sigma_{\sigma}$ as follows.\footnote{The punctured discs appeared in Bessis' work \cite{bessis_dual_2006}. For better visualization, here we prefer to use an alternative description using compact Riemann surfaces with one or two marked points.}
Let $G_1$ and $G_2$ be two identical copies of a regular $n$-gon. Label the edges of each of the two  $n$-gons
 by $T_{\sigma(1)}, T_{\sigma(2)}, \dots , T_{\sigma(n)}$ counter-clockwise. On $G_i$, let $L_i$ be    the line segment from the center of $G_i$ to to the common endpoint of $T_{\sigma(n)}$ and $T_{\sigma(1)}$. Fix the orientation of every edge of $G_1$ (resp.  $G_2$) to be 
 counter-clockwise (resp. clockwise) as in the following picture. 
 
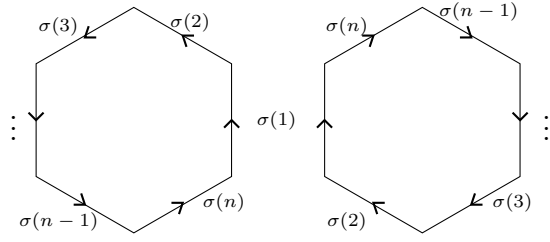
\begin{center}
 \begin{tikzpicture}[scale=0.5]
\node at (2.4,-2.2){\tiny{$\sigma(n)$}};
\node at (1.5,2.6){\tiny{$\sigma(2)$}};
\node at (3.8,0){\tiny{$\sigma(1)$}};
\node at (-2.0,-2.7){\tiny{$\sigma(n-1)$}};
\node at (-2.0,2.5){\tiny{$\sigma(3)$}};
\node at (-3.2,0){\vdots};
\draw (0,0) +(30:3cm) -- +(90:3cm) -- +(150:3cm) -- +(210:3cm) --
+(270:3cm) -- +(330:3cm) -- cycle;
\draw [thick] (2.4,-0.2) -- (2.6,0)--(2.8,-0.2);
\draw [thick] (1.4,1.95) -- (1.3,2.25)--(1.6,2.25);  
\draw [thick] (-1.0,2.2) -- (-1.3,2.2)--(-1.2,2.5);  
\draw [thick] (-2.4,0.2) -- (-2.6,0)--(-2.8,0.2);   
\draw [thick] (1.0,-2.2) -- (1.3,-2.2)--(1.2,-2.5);  
\draw [thick] (-1.4,-1.95) -- (-1.3,-2.25)--(-1.6,-2.25);  
\end{tikzpicture}
\begin{tikzpicture}[scale=0.5]
\node at (2.4,-2.2){\tiny{$\sigma(3)$}};
\node at (1.5,2.9){\tiny{$\sigma(n-1)$}};
\node at (3.3,0){\vdots};
\node at (-2.0,-2.7){\tiny{$\sigma(2)$}};
\node at (-2.0,2.5){\tiny{$\sigma(n)$}};
\draw (0,0) +(30:3cm) -- +(90:3cm) -- +(150:3cm) -- +(210:3cm) --
+(270:3cm) -- +(330:3cm) -- cycle;
\draw [thick] (-2.4,-0.2) -- (-2.6,0)--(-2.8,-0.2);
\draw [thick] (-1.4,1.95) -- (-1.3,2.25)--(-1.6,2.25);  
\draw [thick] (1.0,2.2) -- (1.3,2.2)--(1.2,2.5);  
\draw [thick] (2.4,0.2) -- (2.6,0)--(2.8,0.2);   
\draw [thick] (-1.0,-2.2) -- (-1.3,-2.2)--(-1.2,-2.5);  
\draw [thick] (1.4,-1.95) -- (1.3,-2.25)--(1.6,-2.25);  
\end{tikzpicture}
 \end{center}

 Let $\Sigma_{\sigma}$ be the (compact) Riemann surface of genus $\lfloor \frac{n-1}{2}\rfloor$
obtained by gluing together the two $n$-gons with all the edges of the same label identified according 
to their orientations.  The edges of the $n$-gons become $n$ different curves in $\Sigma_\sigma$. If $n$ is odd, all the vertices of the two $n$-gons 
are identified to become one point in $\Sigma_\sigma$ and the curves obtained from the edges are loops. If $n$ is even, two distinct
 vertices are shared by all curves. Let $\mathcal{T}$ be the set of all curves, i.e., $\mathcal{T}={T}_1\cup\cdots\cup{T}_n\subset \Sigma_\sigma$, and $V$ be the set of the vertex (or vertices) on $\mathcal{T}$.  

 For example, when $n=3$, let $\Sigma_\sigma$ be the closed Riemann surface of genus $1$ with a single marked point $V$, and let $\widetilde{\Sigma_{\sigma}}$ be the universal cover of $\Sigma_\sigma$, which can be regarded as $\mathbb{R}^2$. 
Let $\alpha=\sigma(1)$, $\beta=\sigma(2)$, and $\gamma=\sigma(3)$. Fix three arcs $T_\alpha, T_\beta,$ and $T_\gamma$ on $\Sigma_\sigma$ and the projection $p:\widetilde{\Sigma_{\sigma}}\longrightarrow \Sigma_{\sigma}$ such that  $p^{-1}(T_\alpha)=\mathbb{Z}\times \mathbb{R}\subset \mathbb{R}^2$, $p^{-1}(T_\beta)=\{(x,y) \, : \, x+y\in \mathbb{Z}\}\subset \mathbb{R}^2$, $p^{-1}(T_\gamma)=\mathbb{R}\times \mathbb{Z}\subset \mathbb{R}^2$, and $p^{-1}(V)=\mathbb{Z}^2\subset \mathbb{R}^2$.
Hence $T_\alpha$ is the vertical line segment, $T_\beta$ is the diagonal, and $T_\gamma$ is the horizontal line segment. 
Let ${T}={T}_1\cup T_2 \cup T_3$. See Figure 1.

Let $\mathfrak W$ be the set of words $\mathfrak w=i_1i_2 \cdots i_k$ from the alphabet $I$ such that no two consecutive letters $i_p$ and $i_{p+1}$ are the same.   For each element $w\in W$, let $R_w\subset \mathfrak W$ be the set of words $i_1i_2 \cdots i_k$ such that $w=s_{i_1}s_{i_2} \cdots s_{i_k}$. Recall that the set of positive real roots and the set of reflections in $\mathcal{W}$ are in one-to-one correspondence. Also note that if ${Q}$ is abundunt, i.e., $\mathcal{W}=\langle s_1,...,s_N\ : \ s_1^2=\cdots=s_N^2=e \rangle$, then there is a unique expression for every element $w\in \mathcal{W}$ 
as a product of simple reflections (with no two consecutive simple reflections being the same), hence $R_w$ contains a unique element.  
Note that $$
\mathfrak{R}=\bigcup_{\tiny{\begin{array}{c}w\in \mathcal{W}\\w:\text{reflection}\end{array}}}R_w\subset \mathfrak W.
$$  

\begin{Def} \label{def-adm} 
An \emph{admissible} curve is a continuous function $\eta:[0,1]\longrightarrow \Sigma_{\sigma}$ such that

1) $\eta(x)\in V$ if and only if  $x\in\{0,1\}$;

2) there exists $\epsilon>0$ such that $\eta([0,\epsilon])\subset L_1$ and $\eta([1-\epsilon,1])\subset L_2$;

3) if $\eta(x)\in \mathcal{T}\setminus V$ then $\eta([x-\epsilon,x+\epsilon])$ meets $\mathcal{T}$ transversally for sufficiently small $\epsilon>0$;

4) and   $\upsilon(\eta)\in \mathfrak{R}$, where $\upsilon(\eta):={i_1}\cdots {i_k}$ is given by 
$$\{x\in(0,1) \ : \ \eta(x)\in \mathcal{T}\}=\{x_1<\cdots<x_k\}\quad \text{ and }\quad \eta(x_\ell)\in T_{i_\ell}\text{ for }\ell\in\{1,...,k\}.$$ 

\end{Def}


\begin{figure}\label{Figure-1}
\centering
\begin{tikzpicture}[scale=0.7mm]
\draw [help lines] (0,0) grid (3,2);
\draw [help lines] (0,1)--(1,0);
\draw [help lines] (0,2)--(2,0);
\draw [help lines] (3,0)--(1,2);
\draw [help lines] (3,1)--(2,2);
\foreach \i in {0, ..., 3}
    \foreach \j in {0, ..., 2}{
        \draw (\i,\j)node{$\bullet$};}
\draw[thick, red] (1,1)node[below left]{$V$};
\draw [thick, red] (1.5,1.1)node{$T_\gamma$};
\draw [thick, red] (1.6,1.6)node{$T_\beta$};
\draw [thick, red] (1.1,1.5)node{$T_\alpha$};
\end{tikzpicture}
\caption{This picture illustrates a portion of the universal cover $\Sigma_\sigma$, and the three arcs $T_\alpha, T_\beta,$ and $T_\gamma$.}
\end{figure}
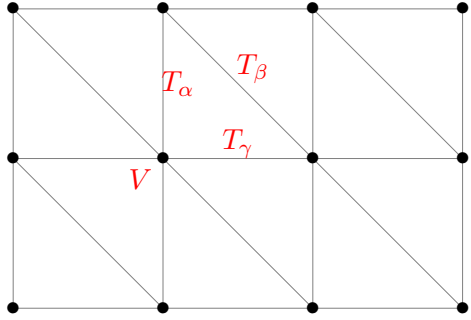


\subsection{Forks and Admissible GIM's} \label{sub-sec-forks-admissible}

\begin{Def}{\cite[Definition 2.10]{seven_cluster_2010}}
    Let $A$ be a GIM such that $|a_{ij}| = |b_{ij}|$ for all $i \neq j$.
    Then $A$ is admissible if it satisfies the following sign condition:
    for any chordless cycle $Z$ in the corresponding quiver $Q$ to $B$, the product $\prod_{(i,j) \in Z} (-a_{ij})$ over the edges of $Z$ is negative if the cycle is oriented and positive otherwise.
\end{Def}

\begin{Def}{\cite[Definition 6.2]{fomin_long_2023}} \label{def-vortex}
    A \textit{vortex} is a rank 4 quiver $Q$ such that
    \begin{itemize}
        \item The underlying simple graph of $Q$ is complete

        \item One of the vertices of $Q$ is a source or sink, called the apex

        \item The quiver $Q$ has an oriented cycle of length 3 (also known as a 3-cycle) as a subquiver.
    \end{itemize}
    For example, all vortices look like 
    $$
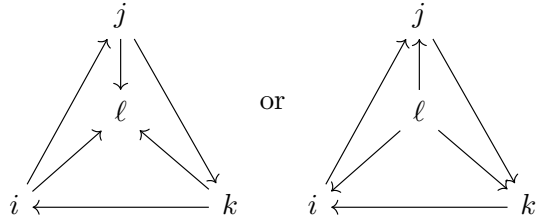
\begin{tikzcd}
        & j \\
        & \ell \\
        i & & k
        \arrow[from=3-1,to=1-2]
        \arrow[from=1-2,to=3-3]
        \arrow[from=3-3,to=3-1]
        \arrow[from=3-1,to=2-2]
        \arrow[from=3-3,to=2-2]
        \arrow[from=1-2,to=2-2]
    \end{tikzcd} 
    \text{ or }
    \begin{tikzcd}
        & j \\
        & \ell \\
        i & & k
        \arrow[from=3-1,to=1-2]
        \arrow[from=1-2,to=3-3]
        \arrow[from=3-3,to=3-1]
        \arrow[from=2-2,to=3-1]
        \arrow[from=2-2,to=3-3]
        \arrow[from=2-2,to=1-2]
    \end{tikzcd} $$
    We say that a quiver is \textit{vortex-free} if it does not have a vortex as a subquiver.
    Note that no quiver with a vortex may admit an admissible GIM.
\end{Def}

\begin{Lem} \label{lem-fork-vortices}
    Let $F$ be a fork with point of return $r$.
    Then $F$ is vortex-free.
\end{Lem}

\begin{proof}
    Suppose that $F$ is not vortex free. Then we may assume that there is some subquiver of $F$ that looks like
    $$
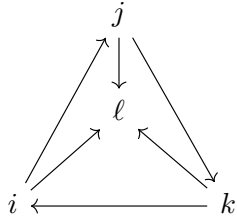
\begin{tikzcd}
        & j \\
        & \ell \\
        i & & k
        \arrow[from=3-1,to=1-2]
        \arrow[from=1-2,to=3-3]
        \arrow[from=3-3,to=3-1]
        \arrow[from=3-1,to=2-2]
        \arrow[from=3-3,to=2-2]
        \arrow[from=1-2,to=2-2]
    \end{tikzcd} $$
    This forces $r \in \{i,j,k\}$.
    However, this would produce vertices in $b \in F^+(r) \cap \{i,j,k, \ell\}$ and $a \in F^-(r) \cap \{i,j,k,\ell\}$ such that $a \to r \to b$ but $a \to b$. 
    A similar argument holds if $\ell$ is a source instead of a sink in the subquiver induced by $\{i,j,k,\ell\}$.
    Therefore, every fork $F$ is vortex-free.
\end{proof}

\begin{Def}{\cite[Definition 1.6]{seven_cluster_2014}} \label{def-GIM-mutation}
    If $A$ is a GIM corresponding to $B$, then the mutation of $A$ at a vertex $k$ is given by
    \[\mu_k(a_{ij}) = \begin{cases}
        \epsilon \sgn(b_{ik})a_{ik}, & \text{ if } j = k \neq i; \\
        -\epsilon \sgn(b_{kj})a_{kj}, & \text{ if } i = k \neq j; \\
        a_{ij} - \sgn(a_{ik}a_{kj})\max(b_{ik}b_{kj},0), & \text{ otherwise},
    \end{cases}\]
    where $\epsilon = -\sgn(c_k)$.
    This definition is equivalent to
    \[\mu_k(a_{ij}) = \begin{cases}
        -\sgn(b_{ik}c_k)a_{ik}, & \text{if } j = k \neq i; \\
        -\sgn(b_{jk}c_k)a_{kj}, & \text{if } i = k \neq j; \\
        a_{ij} - a_{ik}a_{kj}, & \text{if } b_{ik}b_{kj} > 0;\\
        a_{ij}, & \text{if } b_{ik}b_{kj} \leq 0.
    \end{cases}\]
    If $A$ is admissible, then $\mu_k(A)$ is a GIM corresponding to $\mu_k(B)$ \cite[Proposition 2.2]{seven_cluster_2014}.
\end{Def}

\section{Generalized intersection matrices for mutated quivers}

In this section, we will prove Proposition~\ref{lem-fork-admissible-quadratic-form}. In doing so, we will need several lemmas.

\begin{Lem} \label{lem-fork-admissible-ordering}
    Every fork admits at least one ordering such that the corresponding GIM is admissible.
\end{Lem}

\begin{proof}
    Let $F$ be a fork with point of return $r$.
    Then $F \setminus \{r\}$ is an acyclic quiver by definition.
    As such, the GIM $A$ corresponding to $F \setminus \{r\}$ where each entry $a_{ij} = - |f_{ij}|$ is admissible.
    Now, construct the matrix $A$ where $a_{kk} = 2$, $a_{rj} = |f_{rj}|$ for all $j \in F^+(r)$, $a_{ir} = -|f_{ir}|$ for all $i \in F^-(r)$, and $a_{ij} = -|f_{ij}|$ for all $i,j \neq r$.
    The vertices ${i,j,k}$ form a triangle in $F$. 
    Then either it is an undirected triangle or it is directed with $r \in \{i,j,k\}$.
    If it is undirected, then $(-a_{ij})(-a_{jk})(-a_{ki}) = |f_{ij}||f_{jk}||f_{ki}| > 0$ whenever $r \notin \{i,j,k\}$, which makes it positive.
    If $r \in \{i,j,k\}$ and the cycle is undirected, assume without loss of generality that $r = k$.
    Then $i,j \in F^{-}(r)$ or $i,j \in F^{+}(r)$ from the definition of a fork.
    This forces $(-a_{ij})(-a_{jk})(-a_{ki}) = |f_{ij}||f_{jk}||f_{ki}|$ or $(-a_{ij})(-a_{jk})(-a_{ki}) = |f_{ij}|(-|f_{jk}|)(-|f_{ki}|) = |f_{ij}||f_{jk}||f_{ki}|$ respectively.
    In either case, the product is positive as desired.
    
    If the vertices ${i,j,k}$ form a triangle in $F$, we may assume that $k = r$, $i \to r$, and $r \to j$ without loss of generality.
    Then $(-a_{ij})(-a_{rj})(-a_{ir}) = |f_{ij}|(-|f_{rj}|)|f_{ir}| < 0$.
    Thus every chordless cycle is admissible, and this forces $A$ to be an admissible GIM.
\end{proof}

\begin{Cor} \label{cor-fork-admissible-ordering}
    Let $F$ be a fork with $n$ vertices and point of return $r$.
    If $a_1 \prec a_2 \prec \dots \prec a_{n-1}$ is the acyclic ordering  corresponding to $F \setminus \{r\}$, then $r \prec a_{n-1} \prec a_{n-2} \prec \dots  \prec a_{1}$ is the linear ordering which produces the GIM $A$ in Lemma \ref{lem-fork-admissible-ordering}.
    Naturally, any cyclic permutation of this linear ordering also produces a GIM that is admissible.
\end{Cor}

\begin{Rmk} \label{rmk-fork-admissible-ordering}
    The ordering in Corollary \ref{cor-fork-admissible-ordering} is essentially the reverse of the \textit{proper cyclical ordering} of a fork given in \cite{fomin2024cyclicallyorderedquivers}.
\end{Rmk}


\begin{Lem} \label{lem-admissible-complete}
    Let $A$ be an admissible GIM for a quiver $B$ that is complete and vortex-free.
    If $\mu_k(B)$ is both complete and vortex-free, then $\mu_k(A)$ is an admissible GIM of $\mu_k(B)$.
\end{Lem}

\begin{proof}
    Assume that $\mu_k(B)$ is both complete and vortex-free.
    Let $Z$ be a chordless cycle of $Q$.
    Then $Z$ is either a directed 3-cycle or an undirected 3-cycle, as $Q$ is complete.
    Hence, either $-a_{i\ell}a_{\ell j}a_{ji} < 0$ or $-a_{ik}a_{kj}a_{ji} > 0$ respectively.
    We will first show that if $k$ is a vertex of $Z$, then the appropriate product of entries of $\mu_k(A)$ has the correct sign.
    This will allow us to show the same for the cases where $k$ is not a vertex of $Z$.
    
    Assume that $k$ is a vertex of $Z$, and let $i$ and $j$ be the remaining two vertices.
    If $b_{ik}b_{kj} < 0$, then 
    $$-[\mu_k(a_{ik}) \mu_k(a_{kj}) \mu_k(a_{j i})] = -\sgn(b_{ik}c_k) \sgn(b_{jk}c_k) a_{ik}a_{kj}a_{ji}$$
    $$= -\sgn(b_{ik}b_{jk}) a_{ik}a_{kj}a_{ji}$$
    $$= -a_{ik}a_{kj}a_{ji}.$$
    Thus $-[\mu_k(a_{ik}) \mu_k(a_{kj}) \mu_k(a_{j i})] > 0$ as $b_{ik}b_{kj} < 0$ forces $Z$ to be an undirected cycle in both $B$ and $\mu_k(B)$.

    As $b_{ik}b_{kj} \neq 0$ from $B$ being complete, we next test $b_{ik}b_{kj} > 0$.
    This forces
    $$-[\mu_k(a_{ik}) \mu_k(a_{kj}) \mu_k(a_{j i})] = -\sgn(b_{ik}c_k) \sgn(b_{jk}c_k) a_{ik}a_{kj}(a_{ji}-a_{jk}a_{ki})$$
    $$=-\sgn(b_{ik}b_{jk}) a_{ik}a_{kj}(a_{ji}-a_{jk}a_{ki})$$
    $$= a_{ik}a_{kj}a_{ji}-a_{jk}^2a_{ki}^2.$$
    If $Z$ is undirected, then the resulting cycle in $\mu_k(B)$ is directed.
    Further, we have $a_{ik}a_{kj}a_{ji}-a_{jk}^2a_{ki}^2 < 0$ from $a_{ik}a_{kj}a_{ji}<0$.
    If $Z$ is directed, then the resulting cycle in $\mu_k(B)$ is directed if $0 < b_{ji} < b_{ik}b_{kj}$ and undirected if $0 < b_{ik}b_{kj} < b_{ji}$.
    In the first case, this implies that 
    $$a_{ik}a_{kj}a_{ji}-a_{jk}^2a_{ki}^2 < a_{ik}^2a_{kj}^2 - a_{jk}^2a_{ki}^2 = 0. $$
    In the second, we have
    $$a_{ik}a_{kj}a_{ji}-a_{jk}^2a_{ki}^2 > a_{ik}^2a_{kj}^2 - a_{jk}^2a_{ki}^2 = 0.$$
    Either case gives us the appropriate sign, and we do not consider the case where $b_{ji} = b_{ik}b_{kj}$, as that would contradict $\mu_k(B)$ being complete.

    Now, assume that $k$ is not a vertex of $Z$, and let $i,j, \ell$ be the vertices of $Z$.
    If $\sgn(b_{ik}) = \sgn(b_{jk}) = \sgn(b_{\ell k})$, then mutating at $k$ does not affect the cycle $Z$ or the corresponding entries of $\mu_k(A)$.
    Hence,
    $$-[\mu_k(a_{ij}) \mu_k(a_{j \ell}) \mu_k(a_{\ell i})] =-a_{ij}a_{j \ell}a_{\ell i}$$
    as desired.
    
    Assume that $\sgn(b_{ik}) = \sgn(b_{jk}) = -\sgn(b_{\ell k})$.
    Then
    $$-[\mu_k(a_{ik}) \mu_k(a_{kj}) \mu_k(a_{j i})] = -a_{ik}a_{kj}a_{ji} > 0$$
    by our previous argument.
    Additionally, we know that
    $$\sgn(\mu_k(a_{ik}) \mu_k(a_{k \ell}) \mu_k(a_{\ell i})) = -\beta \sgn(\mu_k(a_{jk}) \mu_k(a_{k \ell}) \mu_k(a_{\ell j}))$$
    for some $\beta \in \{-1,1\}$.
    Then $\beta = -1$ implies that both $\{i,k, \ell\}$ and $\{j,k,\ell\}$ share the same cycle type in $\mu_k(B)$.
    Thus
    $$\sgn(\mu_k(a_{ik}) \mu_k(a_{k \ell}) \mu_k(a_{\ell i})\mu_k(a_{jk}) \mu_k(a_{k \ell}) \mu_k(a_{\ell j})) = -\beta$$
    $$\sgn(\mu_k(a_{ik}) \mu_k(a_{\ell i})\mu_k(a_{jk})\mu_k(a_{\ell j})) = -\beta$$
    $$\sgn(\mu_k(a_{ik}) \mu_k(a_{\ell i})\mu_k(a_{jk})\mu_k(a_{\ell j})\mu_k(a_{ik}) \mu_k(a_{kj}) \mu_k(a_{j i})) = \beta$$
    $$\sgn( \mu_k(a_{\ell i})\mu_k(a_{\ell j})\mu_k(a_{j i})) = \beta$$
    $$\sgn( \mu_k(a_{\ell i})\mu_k(a_{i j})\mu_k(a_{j \ell})) = \beta.$$
    If $\beta = -1$, then we know that $\mu_k(b_{j\ell}) \mu_k(b_{\ell i}) < 0$.
    Otherwise, both $\{i,k,\ell\}$ and $\{j,k, \ell\}$ would be differing types of cycles, contradicting our earlier observation.
    Thus $Z$ is an undirected cycle in $\mu_k(B)$, and $- \mu_k(a_{\ell i})\mu_k(a_{i j})\mu_k(a_{j \ell}) > 0$ as desired.

    If $\beta = 1$, then $\mu_k(b_{j\ell}) \mu_k(b_{\ell i}) > 0$.
    When $Z$ is undirected, then $\{i,j,k, \ell\}$ induces a vortex in $\mu_k(B)$, where one of $i$ or $j$ is the apex of this vortex, contradicting our assumption on $\mu_k(B)$.
    If $Z$ is directed, then  $- \mu_k(a_{\ell i})\mu_k(a_{i j})\mu_k(a_{j \ell}) < 0$ as desired. 
    As such, if both $B$ and $\mu_k(B)$ are complete and vortex-free and $A$ is an admissible GIM for $B$, then $\mu_k(A) $ is an admissible GIM for $\mu_k(B)$.
\end{proof}

As an immediate corollary, we may use Lemmas \ref{lem-fork-vortices} and \ref{lem-admissible-complete} to say something in general about forks, which gives us Theorem \ref{thm-linear-ordering} as a consequence.

\begin{Cor} \label{cor-fork-mutation-admissible}
    Let $F$ be a fork with point of return $r$ and $\bw$ be a reduced mutation sequence such that every quiver arrived at is both complete and vortex-free.
    If $A$ is the admissible GIM given by Corollary \ref{cor-fork-admissible-ordering}, then $A^\bw$ is an admissible GIM for $\mu_\bw(F)$.
    In particular, if $F$ has 0-forkless part, then $\bw$ can be any mutation sequence.
\end{Cor}

We now observe how this affects the $l$-vectors, beginning with a corollary of the definition.

\begin{Cor}
    Let $B$ any quiver and $A$ any corresponding GIM.
    For any reduced mutation sequence $\bw$ and $i,j \in \mathcal{I}$, we have that 
    \begin{equation*}
    l_i^{\bw[j]} =
    \begin{cases}
    l_i^{\bw} & \text{if $b_{ij}^{\bw}c_j^{\bw } \leq 0$} \\
    \mathbf{r}_j^{\bw}(l_i^{\bw}) = g_j^{\bw} \mathbf{r}_j (g_j^{\bw})^{-1}(g_i^{\bw}(\alpha_i)) & \text{if  $b_{ij}^{\bw}c_j^{\bw } > 0$}
    \end{cases}.
    \end{equation*}
    
\end{Cor}

\begin{Lem} \label{lem-l-mutation}
    Let $B$ any quiver and $A$ any corresponding GIM.
    For any reduced mutation sequence $\bw$ and $i,j \in \mathcal{I}$, we have that 
    \begin{equation*}
    l_i^{\bw[j]} =
    \begin{cases}
    l_i^{\bw} & \text{if $b_{ij}^{\bw}c_j^{\bw } \leq 0$} \\
    l_i^{\bw} + \beta_{ij}^{\bw}l_j^{\bw} & \text{if  $b_{ij}^{\bw}c_j^{\bw } > 0$}
    \end{cases},
    \end{equation*}
    where $\beta_{ij}^{\bw} \alpha_j = P_j(g_j^{\bw})^{-1}(l_i^{\bw})$, where $P_j = \pi(\mathbf{r}_j) - I$.
\end{Lem}

\begin{proof}
    If $b_{ij}^\bw c_k^\bw \leq 0$, then 
    $$l_i^{\bw[j]} = l_i^\bw,$$
    as $l_i^{\bw[j]} = g_i^{\bw[j]}(\alpha_i)$.
    If $b_{ij}^\bw c_k^\bw \geq 0$, then 
    $$l_i^{\bw[j]} = g_j^{\bw} \mathbf{r}_j(g_j^{\bw})^{-1}g_i^{\bw}(\alpha_i) = g_j^{\bw} \mathbf{r}_j(g_j^{\bw})^{-1}(l_i^\bw).$$
    As $\pi(\mathbf{r}_j) = P_j + I$, we get that
    $$l_i^{\bw[j]} = l_i^\bw + g_j^{\bw}P_j(g_j^{\bw})^{-1}(l_i^\bw).$$
    Since $P_j$ is 0 everywhere but the $j$th row, the product $P_j(g_j^{\bw})^{-1}(l_i^\bw) = \beta_{ij}^\bw \alpha_j$ for coefficient $\beta_{ij}^\bw$.
    Hence, 
    $$l_i^{\bw[j]} = l_i^\bw +  g_j^{\bw}( \alpha_j) = l_i^\bw + \beta_{ij}^\bw l_j^\bw.$$
\end{proof}

\begin{Cor} \label{cor-beta-diagonal}
    Let $B$ any quiver and $A$ any corresponding GIM.
    For any reduced mutation sequence $\bw$ and $i \in \mathcal{I}$, we have that
    $$\beta_{ii}^\bw = P_i(g_i^{\bw})^{-1}(l_i^\bw) = P_i(\alpha_i) = -2.$$
\end{Cor}

\begin{Lem} \label{lem-l-mutation-A}
    Let $B$ be a quiver that is both complete and vortex-free and $A$ be any admissible GIM.
    For any reduced mutation sequence $\bw$ such that every quiver arrived at is both complete and vortex-free and $i,j \in \mathcal{I}$, we have that 
    \begin{equation*}
    l_i^{\bw[j]} =
    \begin{cases}
    l_i^{\bw} & \text{if $b_{ij}^{\bw}c_j^{\bw } \leq 0$} \\
    l_i^{\bw} - a_{ij}^{\bw}l_j^{\bw} & \text{if  $b_{ij}^{\bw}c_j^{\bw } > 0$}
    \end{cases},
    \end{equation*}
    where $a_{ij}^{\bw}$ is the $(i,j)$ entry of $A^\bw$.
    Additionally, $a_{ij}^{\bw} = -P_j(g_j^{\bw})^{-1}(l_i^{\bw})$, where $P_j = \pi(\mathbf{r}_j) - I$.
\end{Lem}

\begin{proof}
    We need only show that $\beta_{ij}^\bw = -a_{ij}^\bw $ for each $\bw$, where the case $\bw = []$ is trivial.
    Assume then that $\beta_{ij}^\bw = -a_{ij}^\bw $ for some mutation sequence $\bw$, and we argue by induction.

    As $\beta_{ij}^{\bw[k]}$ is the coefficient of $\alpha_j$ in $P_j(g_j^{\bw[k]})^{-1}(l_i^{\bw[k]})$, we split into several cases depending on the sign of $b_{jk}^\bw c_k^\bw$ and $b_{ik}^\bw c_k^\bw$, assuming that $k \notin \{i,j\}$.
    If both are less than or equal to zero, then $\beta_{ij}^{\bw[k]} = \beta_{ij}^{\bw}$.
    This implies that $b_{ik}^\bw b_{kj}^\bw \leq 0$, forcing 
    $$a_{ij}^{\bw[k]} = a_{ij}^\bw =  -\beta_{ij}^{\bw} = -\beta_{ij}^{\bw[k]}.$$
    
    If $b_{jk}^\bw c_k^\bw \leq 0$ and $b_{ik}^\bw c_k^\bw  \geq 0$, then 
    $$\beta_{ij}^{\bw[k]} \alpha_j = P_j(g_k^{\bw} \mathbf{r}_k(g_k^{\bw})^{-1}g_j^{\bw})^{-1}(l_i^{\bw})$$
    $$= P_j(g_j^{\bw})^{-1} g_k^{\bw}\mathbf{r}_k(g_k^{\bw})^{-1}(l_i^{\bw})$$
    $$= P_j(g_j^{\bw})^{-1} g_k^{\bw}(P_k + I )(g_k^{\bw})^{-1}(l_i^{\bw})$$
    $$= P_j(g_j^{\bw})^{-1} g_k^{\bw}P_k(g_k^{\bw})^{-1}(l_i^{\bw}) + P_j(g_j^{\bw})^{-1}(l_i^{\bw}) $$
    $$= P_j(g_j^{\bw})^{-1} g_k^{\bw}(\beta_{ik}^\bw \alpha_k) + \beta_{ij}^\bw \alpha_j $$
    $$= \beta_{ik}^\bw P_j(g_j^{\bw})^{-1}(l_k^\bw) + \beta_{ij}^\bw \alpha_j $$
    $$= (\beta_{ik}^\bw \beta_{kj}^\bw + \beta_{ij}^\bw) \alpha_j.$$
    Additionally, we get that $b_{ik}^\bw b_{kj}^\bw \geq 0$.
    Then
    $$a_{ij}^{\bw[k]} = a_{ij}^\bw - a_{ik}^\bw a_{kj}^\bw.$$
    Thus we have that
    $$a_{ij}^{\bw[k]} = a_{ij}^\bw - a_{ik}^\bw a_{kj}^\bw = -(\beta_{ik}^\bw \beta_{kj}^\bw + \beta_{ij}^\bw) = - \beta_{ij}^{\bw[k]}.$$
    If instead $b_{jk}^\bw c_k^\bw \geq 0$ and $b_{ik}^\bw c_k^\bw \leq 0$, then a similar argument follows.

    Now, if $b_{jk}^\bw c_k^\bw \geq 0$ and $b_{ik}^\bw c_k^\bw \geq 0$, then
    $$\beta_{ij}^{\bw[k]} \alpha_j = P_j(g_k^{\bw} \mathbf{r}_k(g_k^{\bw})^{-1}g_j^{\bw})^{-1}(l_i^{\bw} + \beta_{ik}^\bw l_k^\bw)$$
    $$= P_j(g_j^{\bw})^{-1} g_k^{\bw}\mathbf{r}_k(g_k^{\bw})^{-1}(l_i^{\bw}) + \beta_{ik}^\bw P_j(g_j^{\bw})^{-1} g_k^{\bw}\mathbf{r}_k(g_k^{\bw})^{-1}(l_k^\bw)$$
    $$= (\beta_{ik}^\bw \beta_{kj}^\bw + \beta_{ij}^\bw) \alpha_j + \beta_{ik}^\bw P_j(g_j^{\bw})^{-1} g_k^{\bw}(P_k + I)(g_k^{\bw})^{-1}(l_k^\bw)$$
    $$= (\beta_{ik}^\bw \beta_{kj}^\bw + \beta_{ij}^\bw) \alpha_j + \beta_{ik}^\bw [P_j(g_j^{\bw})^{-1} g_k^{\bw}(P_k)(g_k^{\bw})^{-1}(l_k^\bw) + P_j(g_j^{\bw})^{-1}(l_k^\bw)]$$
    $$= (\beta_{ik}^\bw \beta_{kj}^\bw + \beta_{ij}^\bw) \alpha_j + \beta_{ik}^\bw [P_j(g_j^{\bw})^{-1} g_k^{\bw}(\beta_{kk}^\bw \alpha_k) +\beta_{kj}^\bw \alpha_j]$$
    $$= (\beta_{ik}^\bw \beta_{kj}^\bw + \beta_{ij}^\bw) \alpha_j + \beta_{ik}^\bw [\beta_{kk}^\bw P_j(g_j^{\bw})^{-1} ( l_k^\bw) +\beta_{kj}^\bw \alpha_j]$$
    $$= (\beta_{ik}^\bw \beta_{kj}^\bw + \beta_{ij}^\bw) \alpha_j + \beta_{ik}^\bw [\beta_{kk}^\bw \beta_{kj}^\bw \alpha_j +\beta_{kj}^\bw \alpha_j]$$
    $$= (\beta_{ik}^\bw \beta_{kj}^\bw + \beta_{ij}^\bw + \beta_{ik}^\bw\beta_{kk}^\bw \beta_{kj}^\bw  +\beta_{ik}^\bw\beta_{kj}^\bw) \alpha_j$$
    $$= (\beta_{ij}^\bw + (\beta_{kk}^\bw + 2)\beta_{ik}^\bw \beta_{kj}^\bw) \alpha_j.$$
    From Corollary \ref{cor-beta-diagonal} and the inductive hypothesis, we know that $\beta_{kk}^\bw = - a_{kk}^\bw = -2$.
    Thus
    $$\beta_{ij}^{\bw[k]} = \beta_{ij}^\bw.$$
    As $b_{ik}^\bw b_{kj}^\bw < 0$, we get that
    $$a_{ij}^{\bw[k]} = a_{ij}^\bw = - \beta_{ij}^\bw = -\beta_{ij}^{\bw[k]}.$$

    We now deal with the cases of $\beta_{ik}^{\bw[k]}$ and $\beta_{kj}^{\bw[k]}$.
    As $\beta_{ik}^{\bw[k]} = \beta_{ik}^\bw$ whenever $b_{ik}^\bw c_k^\bw \leq 0$ and $\beta_{ik}^{\bw[k]} = -\beta_{ik}^\bw$ whenever $b_{ik}^\bw c_k^\bw \geq 0$, we get in either case that
    $$a_{ik}^{\bw[k]} =  -\sgn(b_{ik}^\bw c_k^\bw) a_{ik}^\bw =\pm a_{ik}^\bw = \mp \beta_{ik}^\bw = - \beta_{ik}^{\bw[k]}.$$
    A similar argument shows that $a_{kj}^{\bw[k]} = - \beta_{kj}^{\bw[k]}$.

    Finally, we always get that
    $$a_{kk}^{\bw[k]} = a_{kk}^\bw = -\beta_{kk}^{\bw} = -\beta_{kk}^{\bw[k]} = -2.$$
    This completes the inductive step, and shows that $a_{ij}^\bw = - P_j(g_j^{\bw})^{-1}(l_i^{\bw}) = -\beta_{ij}^\bw$
\end{proof}

\begin{Prop} \label{lem-fork-admissible-quadratic-form}
    Let $B$ be a quiver that is both complete and vortex-free and $A$ be any admissible GIM.
    For any reduced mutation sequence $\bw$ such that every quiver arrived at is both complete and vortex-free and $i,j \in \mathcal{I}$, we have that 
    $$a_{ij}^\bw = l_i^{\bw} A (l_j^{\bw})^T.$$

In particular, when $i=j$,
$$
2=l_i^{\bw} A (l_i^{\bw})^T,
$$
    which implies that $l_i^{\bw}$ is a solution to the quadratic equation of the form 
\begin{equation}
\sum_{i=1}^n x_i^2 + \sum_{1\leq i<j\leq n} \pm a_{ij} x_i x_j =1.
\end{equation}
\end{Prop}

\begin{proof}
    If $\bw = []$, then the result is trivial.
    Assume then that it holds for some $\bw$, and we will investigate $\bw[k]$.
    To do so, we look at $b_{ik}^\bw c_k^{\bw}$ and $b_{jk}^\bw c_k^{\bw}$, first considering the case when $i = j$.

    If $i = j$, then $b_{ik}^\bw b_{ki}^\bw \leq 0$.
    Thus
    $$a_{ii}^{\bw[k]} = a_{ii}^\bw = 2.$$
    If $b_{ik}^\bw c_k^{\bw} \geq0$, then
    $$l_i^{\bw[k]} A (l_i^{\bw[k]})^T = (l_i^\bw - a_{ik}^\bw l_k^\bw) A (l_i^\bw - a_{ik}^\bw l_k^\bw)^T$$
    $$= l_i^\bw A (l_i^\bw)^T - a_{ik}^\bw l_i^\bw A (l_k^\bw)^T - a_{ik}^\bw l_k^\bw A (l_i^\bw)^T + a_{ik}^\bw a_{ki}^\bw l_k^\bw A (l_k^\bw)^T $$
    $$= l_i^\bw A (l_i^\bw)^T - 2a_{ik}^\bw a_{ik}^\bw + (a_{ik}^\bw)^2 a_{kk}^\bw $$
    $$= l_i^\bw A (l_i^\bw)^T$$
    If $b_{ik}^\bw c_k^{\bw} \leq 0$, then
    $$l_i^{\bw[k]} A (l_i^{\bw[k]})^T = l_i^\bw A (l_i^\bw)^T.$$
    In either case, we get that 
    $$a_{ii}^{\bw[k]} = a_{ii}^\bw = 2 = l_i^\bw A (l_i^\bw)^T = l_i^{\bw[k]} A (l_i^{\bw[k]})^T.$$

    Thus we may now assume $i \neq j$.
    If $i = k$ and $b_{kj} \neq 0$, then $b_{kk}^\bw = 0$ and 
    $$a_{kj}^{\bw[k]} = -\sgn(b_{jk}^\bw c_k^\bw) a_{kj}^\bw.$$
    We know that $b_{kk}^\bw c_k^{\bw} = 0$. 
    Thus we only need look at $b_{jk}^\bw c_k^{\bw}$.
    
    If $b_{jk}^\bw c_k^{\bw} \geq0$, then
    $$l_k^{\bw[k]} A (l_j^{\bw[k]})^T = l_k^\bw A (l_j^\bw - a_{jk}^\bw l_k^\bw)^T$$
    $$= l_k^\bw A (l_j^\bw)^T - a_{jk}^\bw l_k^\bw A (l_k^\bw)^T$$
    $$= a_{kj}^\bw - 2 a_{jk}^\bw = -a_{kj}^\bw$$
    This then forces 
    $$a_{kj}^{\bw[k]} = -a_{kj}^\bw = l_k^{\bw[k]} A (l_j^{\bw[k]})^T $$
    If $b_{jk}^\bw c_k^{\bw} \leq 0$, then
    $$l_k^{\bw[k]} A (l_j^{\bw[k]})^T = l_k^\bw A (l_j^\bw)^T$$
    This then forces
    $$a_{kj}^{\bw[k]} = a_{kj}^\bw = l_k^\bw A (l_j^\bw)^T = l_k^{\bw[k]} A (l_j^{\bw[k]})^T $$
    If $j = k$, then a similar argument results in 
    $$a_{ik}^{\bw[k]} = l_i^{\bw[k]} A (l_k^{\bw[k]})^T.$$
    
    We may now assume that $i \neq j$ and $k \notin \{i,j\}$.
    If $b_{ik}^\bw c_k^{\bw} \geq0$ and $b_{jk}^\bw c_k^{\bw} \geq0$, then $b_{ik}^\bw b_{kj}^\bw < 0$.
    Thus
    $$a_{ij}^{\bw[k]} = a_{ij}^\bw$$
    and
    $$l_i^{\bw[k]} A (l_j^{\bw[k]})^T = (l_i^\bw - a_{ik}^\bw l_k^\bw) A (l_j^\bw - a_{jk}^\bw l_k^\bw)^T$$
    $$= l_i^\bw A (l_j^\bw)^T  - a_{jk}^\bw l_i^\bw A (l_k^\bw)^T - a_{ik}^\bw l_k^\bw A (l_j^\bw)^T + a_{ik}^\bw a_{jk}^\bw l_k^\bw A (l_k^\bw)^T$$
    $$= a_{ij}^\bw - a_{jk}^\bw a_{ik}^\bw - a_{ik}^\bw a_{kj}^\bw+ a_{ik}^\bw a_{jk}^\bw a_{kk}^\bw$$
    $$= a_{ij}^\bw - 2a_{ik}^\bw a_{kj}^\bw+ 2a_{ik}^\bw a_{kj}^\bw = a_{ij}^\bw.$$
    Hence,
    $$a_{ij}^{\bw[k]} = l_i^{\bw[k]} A (l_j^{\bw[k]})^T.$$

    If $b_{ik}^\bw c_k^{\bw} \leq 0$ and $b_{jk}^\bw c_k^{\bw} \leq 0$, then $b_{ik}^\bw b_{kj}^\bw < 0$ again.
    Thus 
    $$a_{ij}^{\bw[k]} = a_{ij}^\bw = l_i^\bw A (l_j^\bw)^T = l_i^{\bw[k]} A (l_j^{\bw[k]})^T.$$

    If $b_{ik}^\bw c_k^{\bw} \geq0$ and $b_{jk}^\bw c_k^{\bw} \leq 0$, then $b_{ik}^\bw b_{kj}^\bw > 0$.
    Thus
    $$a_{ij}^{\bw[k]} = a_{ij}^\bw - a_{ik}^\bw a_{kj}^\bw$$
    and 
    $$l_i^{\bw[k]} A (l_j^{\bw[k]})^T = (l_i^\bw - a_{ik}^\bw l_k^\bw) A (l_j^\bw)^T$$
    $$= l_i^\bw A (l_j^\bw)^T  - a_{ik}^\bw l_k^\bw A (l_j^\bw)^T$$
    $$= a_{ij}^\bw - a_{ik}^\bw a_{kj}^\bw.$$
    Hence,
    $$a_{ij}^{\bw[k]} = l_i^{\bw[k]} A (l_j^{\bw[k]})^T.$$
    A similar argument gives the case where $b_{ik}^\bw c_k^{\bw} \leq 0$ and $b_{jk}^\bw c_k^{\bw} \geq 0$.

    Finally, as $B^\bw$ is complete, we do not have $b_{ik}^\bw c_k^\bw = 0$ or $b_{jk}^\bw c_k^\bw = 0$.
    Therefore, this completes the inductive step and proves that
    $$a_{ij}^\bw = l_i^\bw A (l_j^\bw)^T.$$
\end{proof}

\section{$C$-Vectors And $L$-Vectors}

To prove Theorem \ref{thm-l-equals-c}, we will work through another inductive argument. In this section, we will prove Theorem~\ref{thm-a-equal-b-sign}, which implies Theorem \ref{thm-l-equals-c}. 
In this proof, we will find several situations where it is desirable to disallow certain situations on the signs of the $c$-vectors.
As such, we will start by classifying what the signs of these $c$-vectors look like, beginning with Proposition \ref{prop-control-signs-fork}.

\begin{Prop}{\cite{ervin_all_2024}} \label{prop-control-signs-fork}
    Let $B$ be a fork and $\bw$ a non-trivial fork-preserving mutation sequence.
    If $r$ is the last mutation of $\bw$, then 
    \begin{itemize}
        \item If $b_{rj}^\bw c_j^\bw \geq 0$ for some vertex $j \neq r$, then either $\sgn(c_j^{\bw}) = \sgn(c_r^\bw)$ or $ \sgn(c_j^\bw) c_j^\bw \geq \sgn(c_r^\bw) c_r^\bw$;

        \item If $b_{rj}^\bw c_j^\bw \leq 0$ for some vertex $j \neq r$, then $\sgn(c_j^{\bw}) = -\sgn(c_r^\bw)$;
        
        \item If $b_{ij}^\bw c_j^\bw \geq 0$ for distinct vertices $i$ and $j$ that are neither $r$, then $\sgn(c_i^{\bw}) = \sgn(c_j^\bw)$.
    \end{itemize}
\end{Prop}

From this, we have an immediate corollary on the case where $b_{ir}^{\bw} c_r^\bw \geq 0$.

\begin{Cor} \label{cor-signs-por}
    Let $B$ be a fork and $\bw$ a non-trivial fork-preserving mutation sequence.
    If $r$ is the last mutation of $\bw$---hence the point of return of $B^\bw$---and $b_{ir}^{\bw} c_r^\bw \geq 0$ for some vertex $i \neq r$, then $\sgn(c_i^\bw) = -\sgn(c_r^\bw)$ and $\sgn(c_i)^{\bw} c_i^{\bw} \geq \sgn(c_r^\bw) c_r^\bw$. 
\end{Cor}

Corollary \ref{cor-signs-por} leads to some results on the possible configurations the signs of three different $c$-vectors can take during a fork-preserving mutation sequence.

\begin{Lem} \label{lem-signs-equal}
    Let $B$ be a fork and $\bw$ a non-trivial fork-preserving mutation sequence.
    Then
    $$\sgn(c_i^{\bw}) = \sgn(c_j^{\bw}) = \sgn(c_k^{\bw})$$
    for three distinct vertices $\{i,j,k\}$ only when the subquiver induced by $\{i,j,k\}$ does not form an oriented 3-cycle.
\end{Lem}

\begin{proof}
    First assume that $\{i,j,k\}$ induces an oriented 3-cycle as a subquiver of $B^\bw$.
    Then that implies that $r \in \{i,j,k\}$, as $B^\bw$ is a fork with point of return $r$.
    Assume without loss of generality that $r = k$.
    Then 
    $$\sgn(c_i^{\bw}) = \sgn(c_j^{\bw}) = \sgn(c_k^{\bw})$$
    implies that at least one of $b_{ik}^{\bw} c_k^\bw$ and $b_{jk}^{\bw} c_k^\bw$ is a positive quantity.
    This contradicts the equality of all three signs by Corollary \ref{cor-signs-por}, proving the result.
\end{proof}

We then have one final condition on the signs of the $c$-vectors that will be useful in our proof of Theorem \ref{thm-l-equals-c}.

\begin{Lem} \label{lem-signs-i-j-equal-k-not}
    Let $B$ be a fork and $\bw$ a non-trivial fork-preserving mutation sequence.
    If
    $$\sgn(c_i^{\bw}) = \sgn(c_j^{\bw}) = -\sgn(c_k^{\bw})$$
    and $b_{ij}^\bw c_j^\bw \leq 0$ for three distinct vertices $\{i,j,k\}$, then at least one of the following holds:
    \begin{itemize}
        \item $b_{ik}^\bw b_{kj}^\bw < 0$

        \item $b_{ik}^\bw c_k^\bw \leq 0$
    \end{itemize} 
\end{Lem}

\begin{proof}
    Assume that $r$ is the last mutation of $\bw$, making $B^\bw$ a fork with point of return $r$.
    Then $r \notin \{i,k\}$ and $\sgn(c_i^{\bw}) = -\sgn(c_k^\bw)$ would imply that $b_{ik}^\bw c_k^\bw \leq 0$ by Proposition \ref{prop-control-signs-fork}.
    If $i = r$, then $b_{ij}^{\bw} c_j^\bw \leq 0$ would imply that $\sgn(c_i^\bw) = -\sgn(c_j^\bw)$, a clear contradiction.
    As such, we may assume that $k = r$.
    If $b_{ik}^\bw b_{kj}^\bw > 0$, then the definition of a fork ensures that $\sgn(b_{ij}^\bw) = -\sgn(b_{ik}^\bw)$.
    Thus $b_{ik}^{\bw} c_j^{\bw} \geq0$, forcing $b_{ik}^\bw c_k^\bw \leq 0$ as $\sgn(c_j^{\bw}) = -\sgn(c_k^\bw) $ by assumption.
    Therefore, the result is shown.
\end{proof}

Finally, we introduce some new notation to better describe the relationship between $c$-vectors and $l$-vectors before tackling our main theorem.

\begin{Def} \label{def-epsilon-tau}
    Define two $n$-tuples, $\epsilon^{[k]}$ and $\tau^{[k]}$, for each vertex $k \neq r$, where
    $$\epsilon_{i}^{[k]} = \begin{cases}
        1 & \text{if } k \in B^-(r) \text{ and } i \neq k\\
        -1 & \text{if } k \in B^-(r) \text{ and } i = k\\
        -1 & \text{if } k \in B^+(r) \text{ and } i \not \in \{k,r\}\\
        1 & \text{if } k \in B^+(r) \text{ and } i \in \{k,r\}
    \end{cases}$$
    and
    $$\tau_{j}^{[k]} = \begin{cases}
        1 & \text{if } j = r\\
        1 & \text{if } k \in B^-(r)\\
        -1 & \text{if } k \in B^+(r) 
    \end{cases}.$$
    Additionally, for any non-empty mutation sequence $\bw$ not beginning with $r$, we have that
    $$\epsilon_{i}^{\bw[k]} = \begin{cases}
        -\epsilon_i^\bw & \text{if } i = k\\
        \epsilon_i^\bw  & \text{if } i \neq k\\
    \end{cases}$$
    and
    $$\tau_{j}^{\bw[k]} = \tau_j^\bw.$$
\end{Def}

\begin{Thm} \label{thm-a-equal-b-sign}
    Let $B$ be a fork and $A$ the GIM given in Corollary~\ref{cor-fork-admissible-ordering}.
    For any non-trivial fork-preserving mutation sequence $\bw$ and $i,j \in \mathcal{I}$, we have that $l_{ij}^\bw = \epsilon_i^\bw \tau_j^\bw c_{ij}^\bw$ and that
    $$a_{ij}^\bw = \begin{cases}
        2 & \text{if } i = j \\
        -\epsilon_i^\bw \epsilon_j^\bw |b_{ij}^\bw| & \text{if } b_{ij}^\bw c_j^\bw \geq 0 \\
        -\epsilon_i^\bw \epsilon_j^\bw \sgn(c_i^\bw) \sgn(c_j^\bw) |b_{ij}^\bw| & \text{if } b_{ij}^\bw c_j^\bw \leq 0. 
    \end{cases}$$
\end{Thm}

\begin{proof}
    It trivially holds for $\bw = [\ell]$, so we just need to investigate $\bw[k]$.
    Assume then that our result holds for some $\bw$ ending in $r \neq k$.
    If $b_{ik}^\bw c_k^\bw \leq 0$ for $i \neq k$, then $l_i^{\bw[k]} = l_i^\bw$ and $c_i^{\bw[k]} = c_i^\bw$.
    For $i = k$, we have $l_i^{\bw[k]} = l_i^\bw$ and $c_i^{\bw[k]} = -c_i^\bw$.
    In both situations, this would force $l_{ij}^{\bw[k]} = \epsilon_i^{\bw[k]}\tau_j^{\bw[k]} c_{ij}^{\bw[k]}$ by the inductive hypothesis, as desired.
    We just need to check the case where $b_{ik}^\bw c_k^\bw \geq 0$.
    
    If $b_{ik}^\bw c_k^\bw \geq 0$, then $i \neq k$, $l_{ij}^{\bw[k]} = l_{ij}^\bw - a_{ik}^\bw l_{kj}^\bw$, and $c_{ij}^{\bw[k]} = c_{ij}^\bw + |b_{ik}^\bw| c_{kj}^\bw$.
    By the inductive hypothesis, we find that
    $$l_{ij}^{\bw[k]} = \epsilon_i^\bw \tau_j^\bw  c_{ij}^\bw - \epsilon_k^\bw \tau_j^\bw a_{ik}^\bw c_{kj}^\bw$$
    $$ = \epsilon_i^\bw \tau_j^\bw  c_{ij}^\bw + \epsilon_i^\bw \epsilon_k^\bw \epsilon_k^\bw \tau_j^\bw |b_{ik}^\bw| c_{kj}^\bw$$
    $$= \epsilon_i^\bw \tau_j^\bw  (c_{ij}^\bw + \epsilon_k^\bw \epsilon_k^\bw |b_{ik}^\bw| c_{kj}^\bw)$$
    $$= \epsilon_i^\bw \tau_j^\bw  (c_{ij}^\bw + |b_{ik}^\bw| c_{kj}^\bw)$$
    $$= \epsilon_i^{\bw[k]} \tau_j^{\bw[k]}  c_{ij}^{\bw[k]}.$$
    We are then left with the arduous task of proving that
    $$a_{ij}^{\bw[k]} = \begin{cases}
        2 & \text{if } i = j \\
        -\epsilon_i^{\bw[k]} \epsilon_j^{\bw[k]} |b_{ij}^{\bw[k]}| & \text{if } b_{ij}^{\bw[k]} c_j^{\bw[k]} \geq 0 \\
        -\epsilon_i^{\bw[k]} \epsilon_j^{\bw[k]} \sgn(c_i^{\bw[k]}) \sgn(c_j^{\bw[k]}) |b_{ij}^{\bw[k]}| & \text{if } b_{ij}^{\bw[k]} c_j^{\bw[k]} \leq 0 .
    \end{cases}$$

    The case of $i = j$ is already done, as shown in Lemma \ref{lem-fork-admissible-quadratic-form}.
    If $j = k$, we split into two cases.
    Either $b_{ik}^{\bw[k]} c_k^{\bw[k]} = b_{ik}^{\bw} c_k^{\bw} \geq0$ or $b_{ik}^{\bw[k]} c_k^{\bw[k]} = b_{ik}^{\bw} c_k^{\bw} \leq 0$.
    Then
    $$a_{ik}^{\bw[k]} = -\sgn(b_{ik}^\bw c_k^\bw) a_{ik}^\bw.$$
    In the first case, this reduces to
    $$a_{ik}^{\bw[k]} = \sgn(b_{ik}^\bw c_k^\bw) \epsilon_i^\bw \epsilon_k^\bw |b_{ik}^\bw|$$
    $$  = -\epsilon_i^{\bw[k]} \epsilon_k^{\bw[k]} |b_{ik}^{\bw[k]}|$$
    as expected.
    In the second case, we find that
    $$a_{ik}^{\bw[k]} = \sgn(b_{ik}^\bw c_k^\bw) \epsilon_i^\bw \epsilon_k^\bw \sgn(c_i^\bw) \sgn(c_k^\bw) |b_{ik}^\bw|$$
    $$  = - \epsilon_i^{\bw[k]} \epsilon_k^{\bw[k]} \sgn(c_i^{\bw[k]}) \sgn(c_k^{\bw[k]}) |b_{ik}^\bw|,$$
    again, as expected.
    If $i = k$ instead, then a similar argument holds.
    We may now assume that $k \notin \{i,j\}$ and $i \neq j$.
    
    If $b_{ik}^\bw b_{kj}^\bw \leq 0$, then $b_{ij}^{\bw[k]} = b_{ij}^\bw$ and $a_{ij}^{\bw[k]} = a_{ij}^\bw$.
    When $b_{ik}^{\bw} c_k^{\bw} \leq 0 $ and $b_{jk}^{\bw} c_k^{\bw} \leq 0 $, this forces 
    $$a_{ij}^{\bw[k]} =  - \epsilon_i^{\bw} \epsilon_j^{\bw} \sgn(c_i^{\bw}) \sgn(c_j^{\bw}) |b_{ij}^\bw|$$
    $$ =  - \epsilon_i^{\bw[k]} \epsilon_j^{\bw[k]} \sgn(c_i^{\bw[k]}) \sgn(c_j^{\bw[k]}) |b_{ij}^{\bw[k]}|$$
    if $b_{ij}^{\bw[k]} c_j^{\bw[k]} = b_{ij}^\bw c_j^\bw \leq 0$.
    Similarly, if $b_{ij}^{\bw[k]} c_j^{\bw[k]} = b_{ij}^\bw c_j^\bw \geq 0$, then
    $$a_{ij}^{\bw[k]} =  - \epsilon_i^{\bw} \epsilon_j^{\bw} |b_{ij}^\bw|$$
    $$ =  - \epsilon_i^{\bw[k]} \epsilon_j^{\bw[k]} |b_{ij}^{\bw[k]}|,$$
    as desired.
    
    When $b_{ik}^{\bw} c_k^{\bw} \geq0 $ and $b_{jk}^{\bw} c_k^{\bw} \geq0 $ instead, we know that $\sgn(c_i^{\bw[k]}) = \sgn(c_j^{\bw[k]}) = -\sgn(c_k^{\bw[k]})$ by Corollary \ref{cor-signs-por}.
    Additionally, if $\sgn(c_i^{\bw}) = \sgn(c_j^{\bw})$ or $b_{ij}^{\bw} c_j^{\bw} \geq0$, then 
    $$a_{ij}^{\bw[k]} = -\epsilon_i^\bw \epsilon_j^\bw |b_{ij}^{\bw}|$$
    $$ = -\epsilon_i^{\bw[k]} \epsilon_j^{\bw[k]} |b_{ij}^{\bw[k]}|$$
    as desired. 
    If $b_{ij}^{\bw} c_j^{\bw} \leq 0$ and $\sgn(c_i^{\bw}) = -\sgn(c_j^{\bw}) = -\sgn(c_k^{\bw})$, then $b_{kj}^{\bw} c_j^{\bw} \leq 0$, $b_{ki}^{\bw} c_i^{\bw} \geq0$, and $b_{ki}^{\bw} b_{ij}^\bw > 0$, an impossible situation according to Lemma \ref{lem-signs-i-j-equal-k-not}. 
    If $b_{ij}^{\bw} c_j^{\bw} \leq 0$ and $\sgn(c_i^{\bw}) = -\sgn(c_j^{\bw}) = \sgn(c_k^{\bw})$, then $b_{ki}^{\bw} c_i^{\bw} \leq 0$, $b_{kj}^{\bw} c_j^{\bw} \geq0$, and $b_{kj}^{\bw} b_{ji}^\bw > 0$, another impossible situation according to Lemma \ref{lem-signs-i-j-equal-k-not}. 
    Hence, either $\sgn(c_i^\bw) = \sgn(c_j^\bw)$ or $b_{ij}^\bw c_j^\bw \geq 0$, showing our result for $b_{ik}^{\bw} b_{kj}^\bw \leq 0$.

    If $b_{ik}^{\bw} b_{kj}^{\bw} > 0$, then $\sgn(b_{ij}^{\bw[k]}) = \sgn(b_{jk}^{\bw[k]})$ as $k$ is the point of return of $\mu_{\bw[k]}(B)$ and
    $$a_{ij}^{\bw[k]} = a_{ij}^{\bw} - a_{ik}^\bw a_{kj}^\bw.$$
    Assume that $b_{ik}^{\bw} c_k^{\bw} \geq0 $ and $b_{jk}^{\bw} c_k^{\bw} \leq 0 $. 
    If $b_{ij}^{\bw} c_j^{\bw} \geq0$ as well, then 
    $$a_{ij}^{\bw[k]} = -\epsilon_i^\bw \epsilon_j^\bw |b_{ij}^{\bw}| - \epsilon_i^\bw \epsilon_k^\bw \sgn(c_j^\bw) \sgn(c_k^\bw) |b_{ik}^{\bw}| \epsilon_k^\bw \epsilon_j^\bw |b_{kj}^{\bw}|$$
    $$  = -\epsilon_i^\bw \epsilon_j^\bw (|b_{ij}^{\bw}| +   \sgn(c_j^\bw) \sgn(c_k^\bw) |b_{ik}^{\bw} b_{kj}^{\bw}|)$$
    $$  = -\epsilon_i^\bw \epsilon_j^\bw \sgn(c_j^\bw) ( \sgn(c_j^\bw)|b_{ij}^{\bw}| +    \sgn(b_{ik}^\bw) |b_{ik}^{\bw} b_{kj}^{\bw}|)$$
    $$  = -\epsilon_i^\bw \epsilon_j^\bw \sgn(c_j^\bw) b_{ij}^{\bw[k]}$$
    $$  = -\epsilon_i^{\bw[k]} \epsilon_j^{\bw[k]} \sgn(c_j^{\bw[k]}) \sgn(b_{ij}^{\bw[k]})  |b_{ij}^{\bw[k]}|.$$
    If $b_{ij}^{\bw[k]} c_j^{\bw[k]}  \geq 0$, then we are done.
    If $b_{ij}^{\bw[k]} c_j^{\bw[k]} \leq 0$, then $\sgn(c_j^{\bw[k]}) = \sgn(c_k^{\bw[k]})$ from $\sgn(b_{ij}^{\bw[k]}) = \sgn(b_{jk}^{\bw[k]})$ and $b_{jk}^{\bw[k]} c_k^{\bw[k]} \leq  0$.
    As $\sgn(c_i^{\bw[k]}) = -\sgn(c_k^{\bw[k]})$ by Corollary \ref{cor-signs-por}, we find that $\sgn(b_{ij}^{\bw[k]}) = \sgn(c_i^{\bw[k]})$, proving our result.

    If instead $b_{ij}^{\bw} c_j^{\bw} \leq 0$, then
    $$a_{ij}^{\bw[k]} = -\epsilon_i^\bw \epsilon_j^\bw \sgn(c_i^{\bw}) \sgn(c_j^{\bw}) |b_{ij}^{\bw}| - \epsilon_i^\bw \epsilon_k^\bw \sgn(c_j^\bw) \sgn(c_k^\bw) |b_{ik}^{\bw}| \epsilon_k^\bw \epsilon_j^\bw |b_{kj}^{\bw}|$$
    $$  = -\epsilon_i^\bw \epsilon_j^\bw \sgn(c_j^\bw) (\sgn(c_i^{\bw})|b_{ij}^{\bw}| +    \sgn(c_k^\bw) |b_{ik}^{\bw} b_{kj}^{\bw}|)$$
    $$  = -\epsilon_i^\bw \epsilon_j^\bw \sgn(c_j^\bw) ( \sgn(c_i^\bw)|b_{ij}^{\bw}| +    \sgn(b_{ik}^\bw) |b_{ik}^{\bw} b_{kj}^{\bw}|).$$
    If $\sgn(c_i^{\bw}) = \sgn(c_j^{\bw}) = \sgn(c_k^\bw)$, then $\sgn(b_{ij}^{\bw}) = \sgn(b_{jk}^{\bw}) = \sgn(b_{ki}^\bw)$.
    However, this implies that the subquiver induced by $\{i,j,k\}$ is 3-cycle, a contradiction of Lemma \ref{lem-signs-equal}.
    If $\sgn(c_i^{\bw}) = \sgn(c_j^{\bw}) = -\sgn(c_k^\bw)$, then $b_{ij}^{\bw} c_j^\bw \leq 0$, $b_{ik}^{\bw} b_{kj}^{\bw} > 0$, and $b_{ik}^{\bw} c_k^{\bw} \geq0$, a contradiction of Lemma \ref{lem-signs-i-j-equal-k-not}.
    Hence, we know that $\sgn(c_i^\bw) = -\sgn(c_j^\bw)$.
    Then we find that 
    $$ a_{ij}^{\bw[k]} = -\epsilon_i^\bw \epsilon_j^\bw \sgn(c_j^\bw) ( -\sgn(c_j^\bw)|b_{ij}^{\bw}| +    \sgn(b_{ik}^\bw) |b_{ik}^{\bw} b_{kj}^{\bw}|)$$
    $$ = -\epsilon_i^\bw \epsilon_j^\bw \sgn(c_j^\bw) (b_{ij}^{\bw} +    \sgn(b_{ik}^\bw) |b_{ik}^{\bw} b_{kj}^{\bw}|)$$
    $$ = -\epsilon_i^{\bw[k]} \epsilon_j^{\bw[k]} \sgn(c_j^\bw) b_{ij}^{\bw[k]}$$
    $$ = -\epsilon_i^{\bw[k]} \epsilon_j^{\bw[k]} \sgn(c_j^{\bw[k]}) \sgn(b_{ij}^{\bw[k]}) |b_{ij}^{\bw[k]}|.$$
    If $b_{ij}^{\bw[k]} c_j^{\bw[k]} \geq 0$, then we are done.
    If $b_{ij}^{\bw[k]} c_j^{\bw[k]} \leq 0$ and $\sgn(c_i^{\bw[k]}) = -\sgn(c_j^{\bw[k]})$, then we are also done.
    Finally, if $b_{ij}^{\bw[k]} c_j^{\bw[k]} \leq 0$ and $\sgn(c_i^{\bw[k]}) = \sgn(c_j^{\bw[k]})$, then 
    $$\sgn(c_i^{\bw[k]}) = \sgn(c_j^{\bw[k]}) = -\sgn(c_k^{\bw[k]}),$$
    $b_{ik}^{\bw[k]} b_{kj}^{\bw[k]} > 0$, and $b_{ik}^{\bw[k]} c_k^{\bw[k]} \geq 0$, a contradiction of Lemma \ref{lem-signs-i-j-equal-k-not}.
    A similar argument holds when $b_{ik}^{\bw} c_k^{\bw} \leq 0$ and $b_{jk}^{\bw} c_k^{\bw} \geq0 $, completing the inductive step.

    Therefore, for any reduced mutation sequence $\bw$ not beginning with $r$ and $i,j \in \mathcal{I}$, we have that $l_{ij}^\bw = \epsilon_i^\bw \tau_j^\bw c_{ij}^\bw$ and that
    $$a_{ij}^\bw = \begin{cases}
        2 & \text{if } i = j \\
        -\epsilon_i^\bw \epsilon_j^\bw |b_{ij}^\bw| & \text{if } b_{ij}^\bw c_j^\bw \geq 0 \\
        -\epsilon_i^\bw \epsilon_j^\bw \sgn(c_i^\bw) \sgn(c_j^\bw) |b_{ij}^\bw| & \text{if } b_{ij}^\bw c_j^\bw \leq 0. 
    \end{cases}$$
\end{proof}

\section{Proof of Theorem \ref{thm-coxeter-element}}

To prove Theorem \ref{thm-coxeter-element}, we need to first note the careful structure that Proposition \ref{prop-control-signs-fork} gives to a fork and its $c$-vectors.

\begin{Lem} \label{lem-last-green-vertex}
    Let $B$ be a fork and $\bw$ a non-trivial fork-preserving mutation sequence ending in $r$.
    Then either there exists a vertex $v \neq r$ in $B^\bw$ such that
    \begin{itemize}
        \item $\sgn(c_v^{\bw}) > 0$,

        \item $\sgn(c_i^{\bw}) > 0$ for all $i \neq r$ such that $b_{iv}^\bw > 0$,

        \item $\sgn(c_j^{\bw}) < 0$ for all $j \neq r$ such that $b_{vj}^\bw > 0$,
    \end{itemize}
    or $c_r^\bw$ is the only $c$-vector with positive sign.
    
    We refer to $v$ as the \textit{last green vertex} of $B^\bw$.
    If $c_r^\bw$ is the only $c$-vector with positive sign, we refer to $r$ as the last green vertex.
\end{Lem}

\begin{proof}
    First suppose that $\sgn(c_i^\bw) < 0$ for all $i \neq r$. Let $j$ be the source of $B^\bw \setminus \{r\}$. Then $b_{rj}^{\bw} > 0$ by Definition~\ref{def-forks1}, so $b_{rj}^\bw c_j^\bw \leq 0$.
    As such, Proposition \ref{prop-control-signs-fork} tells us that $$\sgn(c_r^\bw) = -\sgn(c_j^\bw) > 0,$$hence $c_r^\bw$ is the only $c$-vector with positive sign. 
    
    Next suppose that there is at least one vertex that is not the point of return with positive $c$-vector.
    Let $v$ be the maximal vertex in the acyclic ordering of $B^\bw$ with positive $c$-vector.
    Then $b_{iv}^\bw > 0$ for some vertex $i \neq r$ implies $i\neq v$, $b_{iv}^\bw c_v^\bw \geq 0$, and 
    $$\sgn(c_i^\bw) = \sgn(c_v^\bw) > 0$$
    by Proposition \ref{prop-control-signs-fork}.
    If $b_{vj}^\bw > 0$ for some vertex $j \neq r$, then $v \prec j$.
    As $v$ is maximal with respect to the acyclic ordering, this forces $\sgn(c_j^\bw) < 0$.
    Thus our result is shown.
\end{proof}

Lemma \ref{lem-last-green-vertex} culminates in the following corollary.

\begin{Cor} \label{cor-ordering-signs-cycle}
    Let $B$ be a fork and $\bw$ a non-trivial fork-preserving mutation sequence ending in $r$.
    Further, let $F = \mu_\bw(B)$ and $r \prec v_{n-1} \prec \dots \prec v_1$ be the linear ordering of $F$ given in Corollary \ref{cor-fork-admissible-ordering}.
    Then there exists a cyclic permutation of the linear ordering of $F$,
    $$v_j \prec v_{j-1} \prec \dots \prec v_1 \prec r  \prec v_{n-1} \prec v_{n-2} \prec \dots \prec v_{j+1},$$
    where all the vertices preceding $r$ have positive $c$-vectors, all the vertices succeeding $r$ have negative $c$-vectors, and at least one of the vertices in $F$ has a positive $c$-vector.
    Equivalently, the vertex $v_j$ is the last green vertex of $F$.
\end{Cor}

A straightforward application of Proposition \ref{prop-control-signs-fork} then gives the following corollary of how the last green vertex is modified under mutation.

\begin{Cor} \label{cor-last-green-vertex-mutation}
    Let $B$ be a fork and $\bw$ a non-trivial fork-preserving mutation sequence ending in $r$.
    Further, let $F = \mu_\bw(B)$ and $r \prec v_{n-1} \prec \dots \prec v_1$ be the linear ordering of $F$ given in Corollary \ref{cor-fork-admissible-ordering}.
    Let $v_j$ be the last green vertex of $F$ given in Lemma \ref{lem-last-green-vertex}---if there are vertices with positive $c$-vector other than the point of return.
    If $v \neq r$ has $\sgn(c_v^\bw) > 0$, then the last green vertex of $F' = \mu_v(F)$ is
    \begin{itemize}
        \item $v_j$ if $v_j \neq v$;

        \item $v_{j-1}$ if $v_j = v$ and $j > 1$;

        \item $r$ if $v_j = v$ and $j = 1$.
    \end{itemize}
    If $v \neq r$ has $\sgn(c_v^\bw) < 0$, then the last green vertex of $F' = \mu_v(F)$ is
    \begin{itemize}
        \item $v_j$ if $F$ has vertices with positive $c$-vector other than the point of return;

        \item $r$ if $b_{rv}^\bw > 0$ and $c_r^\bw$ was the only positive $c$-vector;

        \item $v$ if $b_{rv}^\bw < 0$ and $c_r^\bw$ was the only positive $c$-vector.
    \end{itemize}
\end{Cor}

We then restate a lemma from a previous work by one of the authors, which results in an immediate corollary for our purposes.

\begin{Lem}{\cite[Lemma 3.4]{ervin_all_2024}}\label{lem-fork-mutate-order}
    Let $F$ be a fork with point of return $r$, where $v_1 \prec v_2 \prec \dots \prec v_{n-1}$ is the unique acyclic ordering on $F \setminus \{r\}$.
    For some vertex $v_j$, let $F' = \mu_{v_j}(F)$.
    Then $F' \setminus \{v_j\}$ has unique acyclic ordering $r \prec v_1 \prec \dots \prec v_{j-1} \prec v_{j+1} \prec \dots \prec v_{n-1}$ if $v_j \in F^+(r)$ and $ v_1 \prec \dots \prec v_{j-1} \prec v_{j+1} \prec \dots \prec v_{n-1} \prec r$ if $v_j \in F^-(r) $.
\end{Lem}

\begin{Cor} \label{cor-ordering-permutation-colors}
    Let $B$ be a fork and $\bw$ a non-trivial fork-preserving mutation sequence ending in $r$.
    Further, let $F = \mu_\bw(B)$ and $r \prec v_{n-1} \prec \dots \prec v_1$ be the linear ordering of $F$ given in Corollary \ref{cor-fork-admissible-ordering}.
    Then $F' = \mu_{v_k}(F)$ has linear ordering 
    $v_k \prec v_{n-1} \prec \dots \prec v_{k+1} \prec v_{k-1} \prec \dots \prec v_1 \prec r $ if $v_j \in F^+(r)$  and linear ordering $v_k \prec r \prec v_{n-1} \prec \dots \prec v_{k+1} \prec v_{k-1} \prec \dots \prec v_1 $ if $v_j \in F^-(r) $.
\end{Cor}

We now can proceed to a proof of Theorem \ref{thm-coxeter-element}.

\begin{Thm}[Theorem \ref{thm-coxeter-element}]
    Let $B$ be a fork and $\bw$ a non-trivial fork-preserving mutation sequence.
    We have that 
    $$\mathbf{r}_{\lambda(1)}^\bw ... \mathbf{r}_{\lambda(n)}^\bw = \mathbf{r}_{\rho(1)}...\mathbf{r}_{\rho(n)}$$
    for some permutations $\lambda, \rho \in\mathfrak{S}_n$, where $\lambda(i)$ is determined by the $i$th element of the ordering of $F = \mu_\bw(B)$ given in Corollary \ref{cor-ordering-signs-cycle} and $\rho$ is determined by the first mutation of $\bw$.
\end{Thm}

\begin{proof}
    We argue by induction, beginning with the base case $\bw = [k]$.
    Let $r \prec n-1 \prec \dots \prec 1$ be the linear ordering of $B$ given in Corollary \ref{cor-fork-admissible-ordering}, where $r$ is the point of return of $B$.
    If $i \neq r$, then we naturally have $b_{ik}c_k \geq  0$ if and only if $k \prec i$ in this ordering.
    Additionally, Corollary \ref{cor-ordering-permutation-colors} tells us that either
    $$k \prec n-1 \prec \dots \prec k+1 \prec k-1 \prec \dots \prec 1 \prec r $$
    or
    $$k \prec r \prec  n-1 \prec \dots \prec k+1 \prec k-1 \prec \dots \prec 1$$
    is the linear ordering of $F = \mu_\bw(B)$ given in Corollary \ref{cor-fork-admissible-ordering}, depending on whether $k \in B^+(r)$ or $k \in B^-(r)$.
    If $k \in B^+(r)$, then $b_{rk}c_k \geq 0$.
    As $k$ is the only vertex of $F$ with negative $c$-vector, all other vertices have positive $c$-vectors.
    Thus the cyclic permutation of the linear ordering of $F$ given in Corollary \ref{cor-ordering-signs-cycle} is either
    $$n-1 \prec \dots \prec k+1 \prec k-1 \prec \dots \prec 1 \prec r \prec k$$ 
    or
    $$r \prec  n-1 \prec \dots \prec k+1 \prec k-1 \prec \dots \prec 1 \prec k$$
    respectively.
    Thus 
    $$\mathbf{r}_{n-1}^\bw \dots \mathbf{r}_{k+1}^\bw \mathbf{r}_{k-1}^\bw \dots \mathbf{r}_1^\bw \mathbf{r}_r^\bw \mathbf{r}_k^\bw = \mathbf{r}_{n-1} \dots \mathbf{r}_{k+1} \mathbf{r}_k \mathbf{r}_{k-1} \dots \mathbf{r}_1 \mathbf{r}_r$$
    or 
    $$\mathbf{r}_r^\bw \mathbf{r}_{n-1}^\bw \dots \mathbf{r}_{k+1}^\bw \mathbf{r}_{k-1}^\bw \dots \mathbf{r}_1^\bw \mathbf{r}_k^\bw =\mathbf{r}_r \mathbf{r}_{n-1} \dots \mathbf{r}_{k+1} \mathbf{r}_k \mathbf{r}_{k-1} \dots \mathbf{r}_1.$$
    proving the base case.

    To proceed with the inductive step, assume that 
    $$\mathbf{r}_{\lambda(1)}^\bw ... \mathbf{r}_{\lambda(n)}^\bw = \mathbf{r}_{\rho(1)}...\mathbf{r}_{\rho(n)}$$
    for some $\lambda, \rho \in\mathfrak{S}_n$, where $\lambda(i)$ is determined by the $i$th element of the ordering of $F = \mu_\bw(B)$ given in Corollary \ref{cor-ordering-signs-cycle},
    $$j \prec j-1 \prec \dots \prec 1 \prec r  \prec n-1 \prec n-2 \prec \dots \prec j+1,$$
    where $j$ is the last green vertex of $F$---with the possibility of $j = r$.
    Suppose we further mutate at some vertex $k \neq r$.

    If $\sgn(c_k^\bw) > 0$, then $j \preceq k \prec r$.
    Since our linear ordering on $F$ comes from the acyclic ordering on $F \setminus \{r\}$, this would force $b_{ik}^\bw c_k^\bw \geq 0$ for all $i \neq r $ if and only if $k \prec i \prec r$.
    Hence, there is a sub-ordering
    $$k-1 \prec \dots \prec 1 \prec r$$
    consisting of exactly the vertices $i$ such that $\mathbf{r}_{i}^{\bw[k]} = \mathbf{r}_{k}^{\bw} \mathbf{r}_{i}^{\bw} \mathbf{r}_{k}^{\bw}$, where $r$ is included if and only if $b_{rk}^\bw > 0$.
    For the case where $b_{rk}^\bw > 0$,
    $$\mathbf{r}_{k-1}^{\bw[k]}\dots \mathbf{r}_{1}^{\bw[k]} \mathbf{r}_{r}^{\bw[k]}\mathbf{r}_{k}^{\bw[k]} = \mathbf{r}_{k}^\bw \mathbf{r}_{k-1}^\bw \dots \mathbf{r}_{1}^\bw \mathbf{r}_{r}^\bw.$$
    For the case where $b_{rk}^\bw < 0$,
    $$\mathbf{r}_{k-1}^{\bw[k]}\dots \mathbf{r}_{1}^{\bw[k]} \mathbf{r}_{k}^{\bw[k]} \mathbf{r}_{r}^{\bw[k]} = \mathbf{r}_{k}^\bw \mathbf{r}_{k-1}^\bw \dots \mathbf{r}_{1}^\bw \mathbf{r}_{r}^\bw.$$
    These two cases correspond to the linear orderings
    $$j \prec \dots \prec k+1 \prec k-1 \prec \dots \prec 1 \prec r \prec k \prec n-1 \prec \dots \prec j+1$$ 
    and
    $$j \prec \dots \prec k+1 \prec k-1 \prec \dots \prec 1 \prec k \prec r \prec n-1 \prec \dots \prec j+1$$ 
    respectively.
    Corollary \ref{cor-ordering-permutation-colors} and Corollary \ref{cor-last-green-vertex-mutation} then show that each respective ordering is the ordering given in Corollary \ref{cor-ordering-signs-cycle}.
    If $\sigma$ is the permutation sending $i$ to the $i$th element of this linear ordering, then our two equalities of reflections give that
    $$\mathbf{r}_{\sigma(1)}^{\bw[k]} \dots \mathbf{r}_{\sigma(n)}^{\bw[k]} = \mathbf{r}_{\lambda(1)}^\bw ... \mathbf{r}_{\lambda(n)}^\bw = \mathbf{r}_{\rho(1)}...\mathbf{r}_{\rho(n)},$$
    proving the desired result when $\sgn(c_k^\bw) > 0$.
    A similar argument holds when $\sgn(c_k^\bw) < 0$.
    
    Therefore, we have shown by induction that
    $$\mathbf{r}_{\lambda(1)}^\bw ... \mathbf{r}_{\lambda(n)}^\bw = \mathbf{r}_{\rho(1)}...\mathbf{r}_{\rho(n)}$$
    for some permutations $\lambda, \rho \in\mathfrak{S}_n$, where $\lambda(i)$ is determined by the $i$th element of the ordering of $F$ given in Corollary \ref{cor-ordering-signs-cycle} and $\rho$ is determined by the first mutation of $\bw$.
\end{proof}

From our proof of the base case, we have the following description of $\rho$.

\begin{Cor} \label{cor-coxeter-element-right-side}
    Let $B$ be a framing of a fork and $v$ be the first mutation of a fork-preserving mutation sequence $\bw$.
    We have that 
    $$\mathbf{r}_{\lambda(1)}^\bw ... \mathbf{r}_{\lambda(n)}^\bw = \mathbf{r}_{\rho(1)}...\mathbf{r}_{\rho(n)}$$
    for some permutations $\lambda, \rho \in\mathfrak{S}_n$, as in Theorem \ref{thm-coxeter-element}.
    If $$r \prec n-1 \prec n-2 \prec \dots \prec 1$$ is the ordering of $F$ given in Corollary \ref{cor-fork-admissible-ordering}, then the permutation $\rho$ is determined by
    $$\mathbf{r}_{\rho(1)}...\mathbf{r}_{\rho(n)} = \mathbf{r}_{n-1} \mathbf{r}_{n-2} \dots \mathbf{r}_1 \mathbf{r}_r$$
    or
    $$\mathbf{r}_{\rho(1)}...\mathbf{r}_{\rho(n)} = \mathbf{r}_r \mathbf{r}_{n-1} \mathbf{r}_{n-2} \dots \mathbf{r}_1$$
    depending on whether $v \in B^+(r)$ or $v \in B^-(r)$ respectively.
\end{Cor}

\begin{Cor}[Corollary \ref{cor-non-self-crossing}]
 Let $n$ be any positive integer, and let $Q$ be a fork with $n$ vertices. 
For each fork-preserving mutation sequence $\bw$ from $Q$, there exist pairwise non-crossing and non-self-crossing admissible curves $\eta_i^\bw$ (see Definition~\ref{def-adm})  such that $\mathbf{r}_i^\bw = \nu(\eta_i^\bw)$ for every $i \in \{1,...,n\}$.
\end{Cor}

\begin{proof}
From Theorem \ref{thm-coxeter-element} we know that $\mathbf{r}_{\lambda(1)}^\bw \dots \mathbf{r}_{\lambda(n)}^\bw = \mathbf{r}_{\rho(1)}\dots \mathbf{r}_{\rho(n)}$ is a Coxeter element. 
In \cite{nguyen_proof_2022}, they showed that if the product of $n$ reflections is a Coxeter element, then each reflection and its corresponding root can be represented by a non-self-crossing loop and a non-self-crossing admissible curve in an $n$-punctured disc, respectively. 
Moreover, the order of vertices of the quiver corresponds to the order of its Coxeter element. 
Applying this result with the order determined by the permutation ${\rho}$, the corollary is proved.
\end{proof}

\section{Examples}\label{ex_sec_first}
In Sections~\ref{ex_sec_first} and~\ref{ex_sec_second}, we will consider the following two quivers to demonstrate our theorems: 
the Markov quiver
    \[M = \begin{tikzcd}
    & 2\\
    1 & & 3
    \arrow[from=2-1, to=1-2, "2"]
    \arrow[from=1-2, to=2-3, "2"]
    \arrow[from=2-3, to=2-1, "2"]
    \end{tikzcd}\]
    and 
    \[Q = \begin{tikzcd}
    & 2\\
    1 & & 3
    \arrow[from=2-1, to=1-2, "2"]
    \arrow[from=1-2, to=2-3, "3"]
    \arrow[from=2-3, to=2-1, "3"]
    \end{tikzcd}.\]
Both quivers are mutation-cyclic \cite{BeinekeBrustleHille}. In this section we will consider three mutation sequences, namely, $\bw=[1]$,  $\bw=[1,2]$, and $\bw=[1,2,3]$.

\begin{Exa} \label{exa-first-theorem}
    An example of Theorem~\ref{thm-base-case} is given in the table below:   
    \begin{center}
    \begin{tabular}{|c | c | c|}
    \hline
         Mutation Sequence & $[B^\bw| C^\bw]$-matrix for $Q$ & $[B^\bw| C^\bw]$-matrix for $M$  \\
         \hline
         $\bw = [1]$     & $\begin{bmatrix}[ccc|ccc] 
    0 & - 2 & 3  & -1 & 0 & 0 \\
    2 & 0 & -3   & 0 & 1 & 0\\
    -3 & 3 & 0   & 3 & 0 & 1
    \end{bmatrix}$&  $\begin{bmatrix}[ccc|ccc] 
    0 & - 2 & 2  & -1 & 0 & 0 \\
    2 & 0 & -2   & 0 & 1 & 0\\
    -2 & 2 & 0   & 2 & 0 & 1
    \end{bmatrix}$\\
    \hline
         $\bw=[1,2]$   & $\begin{bmatrix}[ccc|ccc] 
    0 &  2 & -3  & -1 & 0 & 0 \\
    -2 & 0 & 3   & 0 & -1 & 0\\
    3 & -3 & 0   & 3 & 3 & 1
    \end{bmatrix}$&  $\begin{bmatrix}[ccc|ccc]  
    0 &  2 & -2  & -1 & 0 & 0 \\
    -2 & 0 & 2   & 0 & -1 & 0\\
    2 & -2 & 0   & 2 & 2 & 1
    \end{bmatrix}$\\
    \hline
         $\bw=[1,2,3]$& $\begin{bmatrix}[ccc|ccc] 
    0 &  -7 & 3  & -1 & 0 & 0 \\
    7 & 0 & -3   & 9 & 8 & 3\\
    -3 & 3 & 0   & -3 & -3 & -1
    \end{bmatrix}$ &  $\begin{bmatrix}[ccc|ccc] 
    0 & - 2 & 2  & -1 & 0 & 0 \\
    2 & 0 & -2   & 4 & 3 & 2\\
    -2 & 2 & 0   & -2 & -2 & -1
    \end{bmatrix}$\\
    \hline
    \end{tabular}
    \end{center}

    For each quiver, the sign vector of the $C$-matrix for $\bw=[1]$ (resp. $\bw=[1,2]$ and $\bw=[1,2,3]$) is $(-1,1,1)$ (resp. $(-1,-1,1)$ and $(-1,1,-1)$).
\end{Exa}

\begin{Exa} \label{exa-second-corollary}
    We demonstrate Corollary~\ref{cor-reflection-triple}.
    If we mutate both of $Q$ and $M$ with $\bw = [1]$, then we arrive at 
    $$\aligned 
    \mathbf{r}^\bw_1&=\mathbf{r}_1\\
    \mathbf{r}^\bw_2&=\mathbf{r}_2\\
    \mathbf{r}^\bw_3&=\mathbf{r}_1\mathbf{r}_3\mathbf{r}_1.
\endaligned$$
        If we mutate both of them with $\bw=[1,2]$, then we arrive at 
      $$\aligned 
    \mathbf{r}^\bw_1&=\mathbf{r}_1\\
    \mathbf{r}^\bw_2&=\mathbf{r}_2\\
    \mathbf{r}^\bw_3&=\mathbf{r}_2\mathbf{r}_1\mathbf{r}_3\mathbf{r}_1\mathbf{r}_2.
\endaligned$$   
    If we mutate both of them with $\bw=[1,2,3]$, then we arrive at
    $$\aligned 
    \mathbf{r}^\bw_1&=\mathbf{r}_1\\
    \mathbf{r}^\bw_2&=\mathbf{r}_2\mathbf{r}_1\mathbf{r}_3\mathbf{r}_1\mathbf{r}_2\mathbf{r}_1\mathbf{r}_3\mathbf{r}_1\mathbf{r}_2\\
    \mathbf{r}^\bw_3&=\mathbf{r}_2\mathbf{r}_1\mathbf{r}_3\mathbf{r}_1\mathbf{r}_2.
\endaligned$$
\end{Exa}

\begin{Exa} \label{exa-fourth-theorem}
    We demonstrate Theorem~\ref{thm-coxeter-element}.
    If we mutate both of them with $\bw=[1]$, then 
    $$\mathbf{r}_1^\bw \mathbf{r}_3^\bw \mathbf{r}_2^\bw = \mathbf{r}_3 \mathbf{r}_1 \mathbf{r}_2.$$
    If we mutate both of them with $\bw=[1,2]$, then we arrive at 
    $$\mathbf{r}_1^\bw \mathbf{r}_2^\bw \mathbf{r}_3^\bw = \mathbf{r}_3 \mathbf{r}_1 \mathbf{r}_2.$$
    If we mutate both of them with $\bw=[1,2,3]$, then we arrive at
    $$\mathbf{r}_1^\bw \mathbf{r}_3^\bw \mathbf{r}_2^\bw = \mathbf{r}_3 \mathbf{r}_1 \mathbf{r}_2.$$
\end{Exa}

\begin{Exa}
We demonstrate Corollary~\ref{cor-non-self-crossing}.
    If we mutate both of $Q$ and $M$ with $\bw = [1]$, we get the admissible curves as below. 
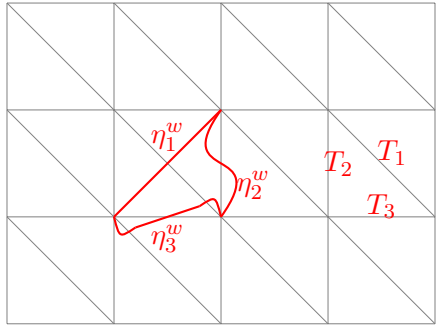
\begin{center}
\begin{tikzpicture}[scale=0.5mm]
\draw [help lines] (-1,-1) grid (3,2);
\draw [help lines] (0,-1)--(-1,0);
\draw [help lines] (1,-1)--(-1,1);
\draw [help lines] (2,-1)--(-1,2);
\draw [help lines] (3,-1)--(0,2);
\draw [help lines] (3,0)--(1,2);
\draw [help lines] (3,1)--(2,2);
\draw [thick,red] (0,0)--(1,1);  
\draw [thick, red] (1,0).. controls (1.5,0.7) and (0.5,0.3) ..(1,1); 
\draw [thick,red] (0.2,-0.1)--(0.8,0.1);  
\draw [thick, red] (0.8,0.1).. controls (0.95,0.2) .. (1,0); 
\draw [thick, red] (0.2,-0.1).. controls (0.05,-0.2) .. (0,0); 
\draw [thick, red] (0.5,0.75)node{$\eta_1^w$};
\draw [thick, red] (1.3,0.3)node{$\eta_2^w$};
\draw [thick, red] (0.5,-0.2)node{$\eta_3^w$};
\draw [thick, red] (2.5,0.1)node{$T_3$};
\draw [thick, red] (2.6,0.6)node{$T_1$};
\draw [thick, red] (2.1,0.5)node{$T_2$};
\end{tikzpicture}
\end{center}
If we mutate both of them with $\bw=[1,2]$, we the admissible curves as below.
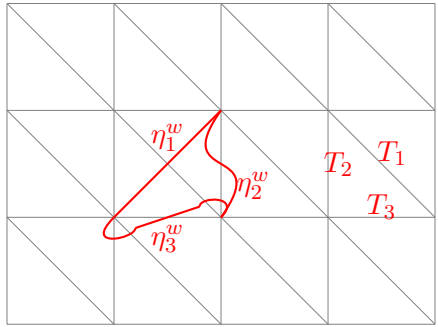
\begin{center}
\begin{tikzpicture}[scale=0.5mm]
\draw [help lines] (-1,-1) grid (3,2);
\draw [help lines] (0,-1)--(-1,0);
\draw [help lines] (1,-1)--(-1,1);
\draw [help lines] (2,-1)--(-1,2);
\draw [help lines] (3,-1)--(0,2);
\draw [help lines] (3,0)--(1,2);
\draw [help lines] (3,1)--(2,2);
\draw [thick,red] (0,0)--(1,1);  
\draw [thick, red] (1,0).. controls (1.5,0.7) and (0.5,0.3) ..(1,1); 
\draw [thick,red] (0.2,-0.1)--(0.8,0.1);  
\draw [thick, red] (0.8,0.1).. controls (0.8,0.2) and (1.2,0.2) .. (1,0); 
\draw [thick, red] (0.2,-0.1).. controls (0.2,-0.2) and (-0.3,-0.3) .. (0,0); 
\draw [thick, red] (0.5,0.75)node{$\eta_1^w$};
\draw [thick, red] (1.3,0.3)node{$\eta_2^w$};
\draw [thick, red] (0.5,-0.2)node{$\eta_3^w$};
\draw [thick, red] (2.5,0.1)node{$T_3$};
\draw [thick, red] (2.6,0.6)node{$T_1$};
\draw [thick, red] (2.1,0.5)node{$T_2$};
\end{tikzpicture}
\end{center}
If we mutate both of them with $\bw=[1,2,3]$, we get the admissible curves as below.
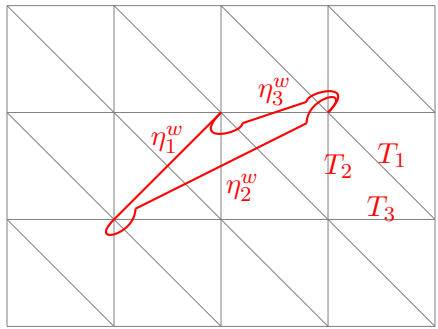
\begin{center}
\begin{tikzpicture}[scale=0.5mm]
\draw [help lines] (-1,-1) grid (3,2);
\draw [help lines] (0,-1)--(-1,0);
\draw [help lines] (1,-1)--(-1,1);
\draw [help lines] (2,-1)--(-1,2);
\draw [help lines] (3,-1)--(0,2);
\draw [help lines] (3,0)--(1,2);
\draw [help lines] (3,1)--(2,2);
\draw [thick,red] (0,0)--(1,1);  
\draw [thick,red] (0.2,0.1)--(1.8,0.9);  
\draw [thick, red] (1.8,0.9).. controls (1.8,1.15) and (2.25,1.25) .. (2,1); 
\draw [thick, red] (0.2,0.1).. controls (0.2,-0.15) and (-0.25,-0.25) .. (0,0); 
\draw [thick,red] (1.2,0.9)--(1.8,1.1);  
\draw [thick, red] (1.8,1.1).. controls (1.8,1.2) and (2.3,1.3) .. (2,1); 
\draw [thick, red] (1.2,0.9).. controls (1.2,0.8) and (0.7,0.7) .. (1,1); 
\draw [thick, red] (0.5,0.75)node{$\eta_1^w$};
\draw [thick, red] (1.2,0.3)node{$\eta_2^w$};
\draw [thick, red] (1.5,1.2)node{$\eta_3^w$};
\draw [thick, red] (2.5,0.1)node{$T_3$};
\draw [thick, red] (2.6,0.6)node{$T_1$};
\draw [thick, red] (2.1,0.5)node{$T_2$};
\end{tikzpicture}
\end{center}

\end{Exa}

\begin{Exa} \label{exa-third-theorem}
    To demonstrate how to calculate $l$-vectors, we consider $l_2^\bw$ for the Markov quiver $M$ with $\bw = [1,2,3]$ and linear ordering $1 \prec 3 \prec 2$.
    First, construct the GIM
    $$A = \begin{bmatrix}
        2 & 2 & -2 \\
        2 & 2 & -2 \\
        -2& -2 & 2
    \end{bmatrix}.$$
    Then consider the following matrices in $M_{3\times 3}(\mathbb{Z})$.
    $$S_1 = \begin{bmatrix}
    -1 & 0 &  0 \\
    -2  & 1 &  0 \\
    2  & 0 &  1
    \end{bmatrix},\quad \quad 
    S_2 = \begin{bmatrix}
    1 & -2 &  0 \\
    0  & -1 &  0 \\
    0  & 2 &  1
    \end{bmatrix},\quad \quad  
    S_3 = \begin{bmatrix}
    1 & 0 &  2 \\
    0  & 1 &  2 \\
    0  & 0 &  -1
    \end{bmatrix}.$$
    From Example \ref{exa-second-corollary} and the definition of $l$-vectors, we know that
    $$\aligned 
    l_2^\bw &= s_2s_1s_3s_1(\alpha_2)\\
    &= (S_2^TS_1^TS_3^TS_1^T (\alpha_2^T) )^T\\
    &= \alpha_2 S_1 S_3 S_1 S_2\\
    &= (- 2\alpha_1 + \alpha_2 )S_3 S_1 S_2\\
    &= (- 2\alpha_1 +\alpha_2  - 2 \alpha_3)S_1 S_2 \\
    &= (- 4\alpha_1 +\alpha_2 - 2 \alpha_3) S_2\\
    &= - 4\alpha_1 + 3\alpha_2  - 2 \alpha_3.
    \endaligned$$
    If we look at $l_2^\bw$ for the quiver $Q$ with $\bw = [1,2,3]$ and linear ordering $1 \prec 3 \prec 2$.
    First, construct the GIM:
    $$A = \begin{bmatrix}
        2 & 2 & -3 \\
        2 & 2 & -3 \\
        -3& -3 & 2
    \end{bmatrix}.$$
    From Example \ref{exa-second-corollary}, we know that
    $$l_2^\bw = s_2s_1s_3s_1(\alpha_2)$$
    $$= s_2s_1s_3(- 2\alpha_1 +\alpha_2 )$$
    $$= s_2s_1(- 2\alpha_1 +\alpha_2  - 3 \alpha_3)$$
    $$= s_2(- 9\alpha_1 +\alpha_2  - 3 \alpha_3)$$
    $$= - 9\alpha_1 + 8\alpha_2  - 3 \alpha_3$$
    
    These calculations can then be used to demonstrate Theorem~\ref{thm-linear-ordering}. Compare the table below with the one given in Example~\ref{exa-first-theorem}.
    
    \begin{center}  
    \begin{tabular}{|c | c | c|}
    \hline
         Mutation Sequence & $L$-matrix for $Q$ & $L$-matrix for $M$  \\
         \hline
         $\bw = [1]$     & $\begin{bmatrix}
    1 & 0 &  0 \\
     0 &   1 &  0 \\
    3 & 0 &  1
    \end{bmatrix}$ & $\begin{bmatrix}
    1 & 0 &  0 \\
     0 &   1 &  0 \\
    2 & 0 &  1
    \end{bmatrix}$\\
    \hline
         $\bw = [1,2]$   & $\begin{bmatrix}
    1 & 0 &  0 \\
     0 &   1 &  0 \\
    3 & -3 &  1
    \end{bmatrix}$ & $\begin{bmatrix}
    1 & 0 &  0 \\
     0 &   1 &  0 \\
    2 & -2 &  1
    \end{bmatrix}$\\
    \hline
         $\bw = [1,2,3]$ & $\begin{bmatrix}
    1 & 0 &  0 \\
     -9 &   8 &  -3 \\
    3 & -3 &  1
    \end{bmatrix}$ & $\begin{bmatrix}
    1 & 0 &  0 \\
     -4 & 3 &  -2 \\
    2 & -2 &  1
    \end{bmatrix}$\\
    \hline
    \end{tabular}
    \end{center}
\end{Exa}

\section{More examples}\label{ex_sec_second}
In this section more interesting examples are given. We still consider the two quivers $Q$ and $M$, but this time the mutation sequence will be different. 

\begin{Exa} \label{exa-first} 

Let $\bw = [1,2,3,2,1,3]$ be our mutation sequence. If $Q$ is 
our initial quiver, then
$$C^\bw = \begin{bmatrix}
433 &  378 &  144 \\
-16  &  -8  &  -3 \\
-24  & -21 &   -8 
\end{bmatrix},$$
giving the sign-vector $(1,-1,-1)$.

Calculating the reflections associated to the mutation sequence and initial quiver produces
$$\mathbf{r}_1^w =
\mathbf{r}_2\mathbf{r}_1\mathbf{r}_3\mathbf{r}_1\mathbf{r}_2\mathbf{r}_1\mathbf{r}_3\mathbf{r}_1\mathbf{r}_2\mathbf{r}_1\mathbf{r}_3\mathbf{r}_1\mathbf{r}_2\mathbf{r}_1\mathbf{r}_2\mathbf{r}_1\mathbf{r}_3\mathbf{r}_1\mathbf{r}_2\mathbf{r}_1\mathbf{r}_3\mathbf{r}_1\mathbf{r}_2\mathbf{r}_1\mathbf{r}_3\mathbf{r}_1\mathbf{r}_2$$
$$\mathbf{r}_2^w = \mathbf{r}_1\mathbf{r}_2\mathbf{r}_1\mathbf{r}_3\mathbf{r}_1\mathbf{r}_2\mathbf{r}_1\mathbf{r}_3\mathbf{r}_1\mathbf{r}_2\mathbf{r}_1$$
$$\mathbf{r}_3^w = \mathbf{r}_2\mathbf{r}_1\mathbf{r}_3\mathbf{r}_1\mathbf{r}_2\mathbf{r}_1\mathbf{r}_3\mathbf{r}_1\mathbf{r}_2\mathbf{r}_1\mathbf{r}_3\mathbf{r}_1\mathbf{r}_2.$$
Taking the product $\mathbf{r}_2^\bw \mathbf{r}_3^\bw \mathbf{r}_1^\bw$ gives us the following as well
$$\mathbf{r}_2^\bw \mathbf{r}_3^\bw \mathbf{r}_1^\bw = \mathbf{r}_3 \mathbf{r}_1 \mathbf{r}_2,$$
demonstrating Theorem \ref{thm-coxeter-element}.

From these reflections, we may calculate the $L$-matrix.
For the linear orderings $ 1 \prec 3 \prec 2$, $2 \prec 1 \prec 3$, or $3 \prec 2 \prec 1$, we have that the $L$-matrix equals the $C$-matrix up to signs of entries, i.e., 
$$L^\bw = \begin{bmatrix}
433 &  -378 &  144 \\
-16  &  8  &  -3 \\
24  & -21 &   8 
\end{bmatrix}; L^\bw = \begin{bmatrix}
433 &  378 &  144 \\
16  &  8  &  3 \\
24  & 21 &   8 
\end{bmatrix}; \text{ or }
L^\bw = \begin{bmatrix}
433 &  378 &  -144 \\
16  &  8  &  -3 \\
-24  & -21 &   8 
\end{bmatrix}$$
For the other linear orderings, we have the following matrix up to signs:
$$L^\bw = \begin{bmatrix}
 283681 & -840402 & -94560\\
   -160 &     80   &    9\\
   -240 &    711   &   80
\end{bmatrix}$$

As an example of how to calculate $L$-vectors, we examine the linear ordering $ 3 \prec 2 \prec 1$.
This gives the corresponding GIM:
$$\begin{bmatrix}
   2 & -2 & 3\\
   -2 & 2 & -3\\
   3 & -3 & 2
\end{bmatrix}.$$
Then the $l$-vector $l_3^\bw$ is found in the following way
$$l^\bw_3 = s_2s_1s_3s_1s_2s_1(\alpha_3) = s_2s_1s_3s_1s_2(-3\alpha_1 + \alpha_3)$$
$$= s_2s_1s_3s_1(-3\alpha_1 - 3\alpha_2 + \alpha_3)$$
$$= s_2s_1s_3(-6\alpha_1 - 3\alpha_2 + \alpha_3)$$
$$= s_2s_1(-6\alpha_1 - 3\alpha_2 + 8\alpha_3)$$
$$= s_2(-24\alpha_1 - 3\alpha_2 + 8\alpha_3)$$
$$= -24\alpha_1 - 21\alpha_2 + 8\alpha_3$$

Finally, below we have the 3 non-self-crossing admissible curves $\eta_i^\bw$ such that $\mathbf{r}_i^\bw = \nu(\eta_i^\bw)$ for every $i \in \{1,2,3\}$.
\begin{center}
\begin{tikzpicture}[scale=0.5mm]
\draw [help lines] (-1,-1) grid (8,6);
\draw [help lines] (-1,0)--(0,-1);
\draw [help lines] (-1,1)--(1,-1);
\draw [help lines] (-1,2)--(2,-1);
\draw [help lines] (-1,3)--(3,-1);
\draw [help lines] (-1,4)--(4,-1);
\draw [help lines] (-1,5)--(5,-1);
\draw [help lines] (-1,6)--(6,-1);
\draw [help lines] (0,6)--(7,-1);
\draw [help lines] (1,6)--(8,-1);
\draw [help lines] (2,6)--(8,0);
\draw [help lines] (3,6)--(8,1);
\draw [help lines] (4,6)--(8,2);
\draw [help lines] (5,6)--(8,3);
\draw [help lines] (6,6)--(8,4);
\draw [help lines] (7,6)--(8,5);
\draw [thick,red] (0.29,0.11)--(6.89,4.87);  
\draw [thick, red] (6.89,4.87).. controls (6.89,5.15) and (7.25,5.25) .. (7,5); 
\draw [thick, red] (0.29,0.11).. controls (0.2,-0.15) and (-0.25,-0.25) .. (0,0); 
\draw [thick,red] (0,0)--(4,3);  
\draw [thick,red] (4.12,3.11)--(6.69,4.89);  
\draw [thick, red] (6.69,4.89).. controls (6.89,5.15) and (7.35,5.35) .. (7,5); 
\draw [thick, red] (4.12,3.11).. controls (4.2,2.85) and (3.75,2.75) .. (4,3); 
\draw [thick, red] (4.6,2.75)node{$\eta_1^w$};
\draw [thick, red] (1.7,1.7)node{$\eta_2^w$};
\draw [thick, red] (5.2,4.2)node{$\eta_3^w$};
\draw [thick, red] (3.5,1.1)node{$T_3$};
\draw [thick, red] (3.6,1.6)node{$T_1$};
\draw [thick, red] (3.1,1.5)node{$T_2$};
\end{tikzpicture}
\end{center}

\end{Exa}


\begin{Exa} \label{exa-second}
Let $\bw = [1,2,3,2,1,3]$ be our mutation sequence. If  $M$ is 
our initial quiver, then
$$C^\bw = \begin{bmatrix}
13 & 8 & 6\\
-6 &-3 &-2\\
-6 &-4 &-3
\end{bmatrix},$$
giving the sign-vector (1,-1,-1), which is identical to the case for the quiver $Q$.

Calculating the reflections associated to the mutation sequence produces
$$\mathbf{r}_1^w =
\mathbf{r}_2\mathbf{r}_1\mathbf{r}_3\mathbf{r}_1\mathbf{r}_2\mathbf{r}_1\mathbf{r}_3\mathbf{r}_1\mathbf{r}_2\mathbf{r}_1\mathbf{r}_3\mathbf{r}_1\mathbf{r}_2\mathbf{r}_1\mathbf{r}_2\mathbf{r}_1\mathbf{r}_3\mathbf{r}_1\mathbf{r}_2\mathbf{r}_1\mathbf{r}_3\mathbf{r}_1\mathbf{r}_2\mathbf{r}_1\mathbf{r}_3\mathbf{r}_1\mathbf{r}_2$$
$$\mathbf{r}_2^w = \mathbf{r}_1\mathbf{r}_2\mathbf{r}_1\mathbf{r}_3\mathbf{r}_1\mathbf{r}_2\mathbf{r}_1\mathbf{r}_3\mathbf{r}_1\mathbf{r}_2\mathbf{r}_1$$
$$\mathbf{r}_3^w = \mathbf{r}_2\mathbf{r}_1\mathbf{r}_3\mathbf{r}_1\mathbf{r}_2\mathbf{r}_1\mathbf{r}_3\mathbf{r}_1\mathbf{r}_2\mathbf{r}_1\mathbf{r}_3\mathbf{r}_1\mathbf{r}_2.$$
Taking the product $\mathbf{r}_2^\bw \mathbf{r}_3^\bw \mathbf{r}_1^\bw$ gives us the following as well
$$\mathbf{r}_2^\bw \mathbf{r}_3^\bw \mathbf{r}_1^\bw = \mathbf{r}_3 \mathbf{r}_1 \mathbf{r}_2.$$
Note that this is exactly what occurred when we looked at the quiver $Q$. 
It further implies that the corresponding admissible curves are the same as in the previous case. 

For the linear orderings, $ 1 \prec 3 \prec 2$, $2 \prec 1 \prec 3$, and $3 \prec 2 \prec 1$, we have that the $L$-matrix equals the $C$-matrix up to signs of entries, i.e., 
$$L^\bw = \begin{bmatrix}
13 & -8 & 6\\
-6 &3 &-2\\
6 &-4 &3
\end{bmatrix}; L^\bw = \begin{bmatrix}
13 & 8 & 6\\
6 &3 &2\\
6 &4 &3
\end{bmatrix}; \text{ or }
L^\bw = \begin{bmatrix}
13 & 8 & -6\\
6 &3 &-2\\
-6 &-4 &3
\end{bmatrix}.$$
For the other linear orderings, we have the following matrix up to signs:
$$L^\bw = \begin{bmatrix} 
 23661 & -68952 & -11830 \\
   -70 &    35  &    6 \\
   -70 &   204  &   35 
\end{bmatrix}.$$

Below, for each of $Q$ and $M$, we compare the $C$-matrix and the $L$-matrix for the linear ordering $2 \prec 1 \prec 3$, demonstrating Theorem~\ref{thm-linear-ordering}. 
 \begin{center}  
    \begin{tabular}{|c | c | c| c | c|}
    \hline
         Mutation Sequence & $C$-matrix for $Q$& $L$-matrix for $Q$ & $C$-matrix for $M$  & $L$-matrix for $M$ \\
    \hline
         $\bw = [1,2,3,2,1,3]$ & $\begin{bmatrix}
433 &  378 &  144 \\
-16  &  -8  &  -3 \\
-24  & -21 &   -8 
    \end{bmatrix}$ & $\begin{bmatrix}
433 &  378 &  144 \\
16  &  8  &  3 \\
24  & 21 &   8 
    \end{bmatrix}$& $\begin{bmatrix}
 13 & 8 & 6\\
 -6 & -3 & -2\\
-6 & -4 & -3
    \end{bmatrix}$& $\begin{bmatrix}
 13 & 8 & 6\\
 6 & 3 & 2\\
6 & 4 & 3
    \end{bmatrix}$\\
    \hline
    \end{tabular}
    \end{center}

\end{Exa}

\bibliographystyle{amsplain}
\bibliography{bibliography} 
\addtocontents{toc}{\protect\setcounter{tocdepth}{1}}


\end{document}